\documentclass[12pt,letterpaper]{amsart}

\usepackage{fancyhdr}
\usepackage{hyperref,xcolor}
\advance \topmargin by -\headsep
     
\evensidemargin \oddsidemargin
\renewcommand{\ge}{\geqslant}
\renewcommand{\le}{\leqslant}

\renewcommand{\Re}{\text{Re}}

\usepackage{tikz,tikz-cd}

\usepackage[mathlines]{lineno} %% <- mathlines turns on line numbering in equations
\usepackage{amsmath}           %% <- N.B., mathtools also loads this
\usepackage{etoolbox} %% <- for \cspreto, \csappto, \patchcmd, \pretocmd, \apptocmd

\newcommand*\linenomathpatch[1]{%
  \cspreto{#1}{\linenomath}%
  \cspreto{#1*}{\linenomath}%
  \csappto{end#1}{\endlinenomath}%
  \csappto{end#1*}{\endlinenomath}%
}
\newcommand*\linenomathpatchAMS[1]{%
  \cspreto{#1}{\linenomathAMS}%
  \cspreto{#1*}{\linenomathAMS}%
  \csappto{end#1}{\endlinenomath}%
  \csappto{end#1*}{\endlinenomath}%
}

\expandafter\ifx\linenomath\linenomathWithnumbers
  \let\linenomathAMS\linenomathWithnumbers
  \patchcmd\linenomathAMS{\advance\postdisplaypenalty\linenopenalty}{}{}{}
\else
  \let\linenomathAMS\linenomathNonumbers
\fi

\linenomathpatch{equation}
\linenomathpatchAMS{gather}
\linenomathpatchAMS{multline}
\linenomathpatchAMS{align}
\linenomathpatchAMS{alignat}
\linenomathpatchAMS{flalign}

\usepackage{todonotes}

\usepackage{amsmath}
\usepackage{amsthm,amssymb,amscd}

\usepackage[color,final]{showkeys}
\definecolor{refkey}{rgb}{0,0.8,0}

\usepackage[inline,final]{showlabels}
\def\showlabelsetlabel#1{\raise1.5ex\hbox{{\showlabelfont #1}}}

\renewcommand{\showlabelfont}{\color{green}}

\makeatletter
\def\SL@eqntext#1{\rlap{\quad\SL@margintext{#1}}}
\makeatother

\makeatletter\def\SL@eqnlefttext #1{\hbox to 0pt{\kern 60pt 
\llap{\SL@margintext{#1}\quad}\hss}}
\makeatother

\newcommand{\p}{{\mathfrak p}}
\newcommand{\PP}{{\mathcal P}}

\newcommand{\Q}{\mathbb{Q}}
\newcommand{\Z}{\mathbb{Z}}

\newcommand{\CC}{\mathbb{C}}

\newcommand{\ind}{{\mathbf 1}}
\newcommand{\eps}{\varepsilon}

\newcommand{\rarr}{{\rightarrow}}

\newcommand{\mymod}[2]{{ #1 \: (\bmod \: {#2})}}

\newcommand{\myvec}[1]{{ \langle #1 \rangle }}

\newcommand{\biquad}{{\Q(\sqrt{d_1}, \sqrt{d_2})}}

\newtheorem{thm}{Theorem}
\newtheorem{lem}{Lemma}

\newtheorem{prop}{Proposition}
\newtheorem{cor}{Corollary}

\theoremstyle{remark}
\newtheorem{quest}{Question}
\newtheorem{remm}{Remark}

\newtheorem*{ack}{Acknowledgments}

\newcommand\gal{{ \mbox{\rm Gal}}}

\newcommand{\trho}{{\tilde{\rho}}}

\newcommand{\ov}[1]{{\overline{{#1}}}}

\DeclareMathOperator{\sign}{sign}
\DeclareMathOperator{\disc}{disc}
\DeclareMathOperator{\odd}{odd}
\DeclareMathOperator{\even}{even}

\newcommand{\flatsum}{\mathop{\sum{\vphantom{\sum}}^\flat}}

\newcommand{\natsum}{\mathop{\sum{\vphantom{\sum}}^\natural}}
\newcommand{\primesum}{\mathop{\sum{\vphantom{\sum}}'}}
\renewcommand{\vec}[1]{\mathbf{#1}}

\DeclareFontFamily{U}{wncyr}{}
\DeclareFontShape{U}{wncyr}{m}{n}{%
   <5> <6> <7> <8> <9> gen * wncyr
   <10> <10.95> <12> <14.4> <17.28> <20.74>  <24.88>wncyr10}{}
\DeclareFontShape{U}{wncyr}{bx}{n}{%
   <5> <6> <7> <8> <9> gen * wncyb
   <10> <10.95> <12> <14.4> <17.28> <20.74>  <24.88>wncyb10}{}
\DeclareSymbolFont{cyr}{U}{wncyr}{m}{n}
\SetSymbolFont{cyr}{bold}{U}{wncyr}{bx}{n}

\DeclareMathSymbol{\sha}{\mathalpha}{cyr}{"58}

\begin{document}

% \boldmath

\title{Counting biquadratic number fields with
  \\
  quaternionic and dihedral extensions}

% \pagestyle{fancy}
% \fancyhead[LO]{{\sc Density of biquadratic fields.}
%   %  \protect \, \, \protect \bf DRAFT \ \mydate\
%   %  \protect \, \protect {\bf DRAFT.} \ \mymonth \the\day \, --
% \sc  \protect \, \ \mymonth \the\day \, --  
%   \ifnum\timehh<10 0\fi\number\timehh\,:\,\ifnum\timemm<10 0\fi\number\timemm
%   %    \, --
%   \,
%   --
%   {\bf \protect\currfiledir\protect\currfilename}
% }
% \fancyhead[LE]{{\sc Louis Gaudet and Siman Wong.  
%   {
% %    \protect \protect\sc\today\ -- 
%     \protect \, \ \mymonth \the\day \, --      
% \ifnum\timehh<10 0\fi\number\timehh\,:\,\ifnum\timemm<10 0\fi\number\timemm
%     \protect \, \, \protect \bf DRAFT
% %    \, --
%     %     {\protect\thepwd\protect\currfilename}
%     \,
%      \protect\currfiledir\protect\currfilename
%   }
% }}
% \fancyhead[R]{\thepage}
% \fancyfoot[C]{}

\pagestyle{fancy}
\fancyhead[LO]{{\sc Louis M.~Gaudet and Siman Wong. \today. }}
\fancyhead[LE]{{\sc Counting biquadratic fields.  \today. }}
\fancyhead[R]{\thepage}
\fancyfoot[C]{}

\author{Louis M.~Gaudet and Siman Wong \\ \today}

%   {
%     \protect \protect\sc\today\ -- 
%     \ifnum\timehh<10 0\fi\number\timehh\,:\,\ifnum\timemm<10 0\fi\number\timemm
%     \protect \, \, \protect \bf DRAFT
%   }
% %  \\    {\protect\thepwd\protect\currfilename}
% %  \\    {\protect\currfiledir\protect\currfilename}  
% }

% \address{Department of Mathematics \& Statistics, University of Massachusetts.
% 	Amherst, MA 01003-9305. USA}

 \begin{abstract}
We establish asymptotic formulae for the number of biquadratic number fields of bounded discriminant that can be embedded into a quaternionic or a dihedral extension.  To prove
these results, we express the solvability of these inverse Galois problems
in terms of Hilbert symbols, and then apply a method of Heath-Brown to bound
sums of linked quadratic characters.
 \end{abstract}

% % \email{siman@math.umass.edu}

% % \thanks{}

% \subjclass{Primary 11N45; Secondary 12F12, 11R20, 11R32, 11F80}

% % \date{\today}

% \keywords{Artin representations, character sums, density,
%   projective representations,  quaternion Galois group,
%   dihedral Galois group, inverse Galois problem}

\maketitle

% \tableofcontents

% {\footnotesize
  % \begin{quote}
    %
   
  % \end{quote}

\section{Introduction}

In this paper we study how often biquadratic number fields can be
embedded in quaternionic and respectively dihedral extensions.
To explain the reason for focusing on these two particular groups, we
first consider this question in the context of the inverse Galois problem.
We then state our main result, give an outline of the proof, and comment on related results. 
In the appendix, 
we revisit this extension problem from the point of view of Galois
representations and discuss further questions that arise.

\subsection{From inverse Galois problems to counting problems}
Let $K$ be a field, and let $G$ be a finite group.  The \textit{inverse Galois problem} (for $G$
over $K$) asks if there exists a Galois extension of $K$ with Galois
group $G$.  Going further, let $N$ be a normal subgroup of  $G$, and let
$
L/K
$
be a finite Galois extension with Galois group $G/N$.  The
\textit{extension problem}
for the pair $(G, L/K)$ asks if there exists a Galois extension of $K$ with 
Galois group $G$ that contains $L$ as an intermediate subfield. If such an extension exists, we will say that $L/K$ admits a $G$-extension. For general
finite groups, the inverse Galois problem is not resolved even for
$
K=\Q
$;
Shafarevich has shown that the inverse Galois problem has a positive answer for solvable groups over any number
field \cite{shaf1, shaf2}.  

Even when the inverse Galois problem is has a positive answer
for $G$ over $K$, the more restrictive extension problem can still fail for a
given
$L/K$.  For example, letting
$
d\not= 1
$
be a square-free integer, we have by \cite[Theorem 1.2.4]{serre:galois} that
\begin{equation}
  \text{$
    \Q(\sqrt{d})
    $
    is contained in a $C_4$-extension of $\Q$ if and only if
    $
    d=u^2 + v^2
    $
    with $u, v\in\Q$.}
     \label{sq}
\end{equation}
Here $C_m$ denotes the cyclic group of order $m$. So while there are asymptotically $cX$ quadratic fields $K$ (up to isomorphism) with $\disc(K)\le X$ (where $c>0$ is an absolute constant), there are only asymptotically $c'X(\log X)^{-1/2}$ of them that admit a $C_4$-extension (here $c'>0$ is also an absolute constant; this asymptotic formula is originally due to Landau \cite{landau}). In particular, the density of quadratic
number fields that admit a $C_4$-extension is zero.

On the other hand, an exercise in class field theory shows that for any prime $p$, every
quadratic extension of
$\Q$
is contained in infinitely many  Galois extensions with Galois group $D_p$, the dihedral group 
of order $2p$ \cite[Theorem I.2.1]{dihedral}.
So the next non-trivial extension problem over $\Q$ involves the quaternion
group
$Q_8$ and the dihedral group $D_4$:  each group has a $C_2$ center with
$
C_2\times C_2
$
quotient, 
and the extension problem now asks if a given biquadratic extension of $\Q$ can
be extended to a $Q_8$- or a $D_4$-extension.
These extension problems are not always solvable; there are classical
criteria for their solvability
due to Witt and others (see Section \ref{sec:criterion}).  In this paper we
study how often these criteria hold.

\begin{remm}
  Here is another reason why the groups
  $
  Q_8, D_4
  $
  are special for extension problems.
For any prime $p$, there are precisely two non-Abelian groups of order
$p^3$  \cite[Exercise~5.3.6]{robinson}.  Both are central extensions of $C_p$
by
$
C_p\times C_p
$.
Suppose $p$ is odd, and let $G$ be one of these non-Abelian groups of order
$p^3$.
For any number field $K$, a standard  Galois cohomology argument shows that
any
$C_p\times C_p$-extension of $K$ can be extended to a Galois extension over
$K$ with Galois group $G$, see~\cite[Lemma~2.1.5]{serre:galois}.  Thus the
enumerative question in this paper is meaningful for non-Abelian groups of order
$p^3$ only when $p=2$.  And the fact that not every biquadratic field can be extended
to a $Q_8$- or $D_4$-extension  means that the $p>2$ hypothesis
in \cite[Lemma~2.1.5]{serre:galois} is critical.
\end{remm}

\subsection{Main result}
By the conductor-discriminant formula, the discriminant of a biquadratic
number field is the product of the discriminants of its quadratic subfields.
Thus we expect the number of biquadratic fields of discriminant
$
\le X
$
to be
$
\sim c \sqrt{X}\log^2 X
$
for some constant $c>0$.  This was made precise by Baily \cite[Theorem 8]{baily},
but there was a mistake in his calculation of the asymptotic constant.  This
was rectified by M\"aki \cite{maki}; see also Cohen--Diaz~y~Diaz--Olivier
     \cite[p.~582]{cohen-disc} and Rome \cite[Theorem 1]{rome}.
To state this result, let us say that a biquadratic field is of sign~$+$ if it is totally real, otherwise
we will say that it is
of sign~$-$, and for $\sigma\in\{\pm\}$ we define
\[
  B^\sigma(X)\;:=\;\#\big\{\text{biquadratic fields of discriminant $\le X$ and sign $\sigma$}\big\}.
\]
Then we have
\begin{equation}
    B^\sigma(X)
    \;=\;
    \frac{23     c^\sigma}
         {960}
    \sqrt{X} (\log X)^2
    \prod_p
      \Bigl(         
         \Bigl(        
           1 - \frac{1}{p}
         \Bigl)^3
         \Bigl(        
           1 + \frac{3}{p}
         \Bigl)
      \Bigl)
    \;+\;
    O(\sqrt{X} \log X),
                                  \label{px}         
\end{equation}
where
$
c^+ = 1/4
$
and
$
c^- = 3/4
$.

Now for $G$ being either $Q_8$ or $D_4$, we define
\[
  B^\sigma(X;G)\;:=\;\#\bigg\{ \!
    \begin{array}{c}
        \text{biquadratic fields of discriminant $\le X$}
        \\
        \text{and sign $\sigma$ that admit a $G$-extension}
      \end{array}\!
  \bigg\}.
\]
Our main result in this paper is the following.

\begin{thm}
  \label{thm:density}
For
$
\sigma\in\{\pm\}
$
and $G=Q_8$ or $D_4$, we have
  \[
  {B}^\sigma(X;G)
  \;=\;
  \frac{c^\sigma(G)}{\sqrt{2\pi}}\sqrt{X}(\log X)^{1/2}
    \prod_p
      \Bigl(
        \Bigl(
          1 - \frac{1}{p}
        \Bigr)^{3/2}
        \Bigl(
          1 + \frac{3}{2p}
        \Bigr)
      \Bigr)
    \;+\;
    O\bigl(  \sqrt{X} (\log X)^{1/4}\bigr),
  \]
where
\[
   c^+(Q_8) = \frac{25}{168},
   \quad
   c^-(Q_8) = 0,
   \quad
   c^+(D_4) = \frac{33}{56},
   \quad\text{and}\quad
   c^-(D_4) = \frac{33}{28}.
\]
\end{thm}

In particular, this theorem shows that the density of biquadratic number fields that admit a $G$-extension is zero for both $G=Q_8$ and $D_4$.

\subsection{Outline of the paper}
\label{sec:questions}
In Section \ref{sec:parametrize} we parameterize biquadratic fields in a way (via ``admissible triples'' of squarefree integers) that will be amenable for
subsequent analysis.
In Section \ref{sec:criterion} we first recall the
classical criteria for a given biquadratic field in characteristic $\not=2$
to admit a 
$Q_8$- and respectively $D_4$-extension, and then with $\Q$ as the ground
field, we
rewrite these criteria
in terms of quadratic symbols.  In Section \ref{sec:char} we use these
criteria to develop formulae for characteristic functions that detect when a given field admits a $G$-extension.
In Section \ref{sec:average}, we apply these characteristic functions to the
parametric families of biquadratic fields from Section \ref{sec:parametrize},
turning the counting problem in Theorem \ref{thm:density} into sums of
``linked'' quadratic characters.
Character sums of similar shape appear in the work of Heath-Brown on the average size
of Selmer groups for the congruent number problem \cite{heath}. In the rest of
the paper, we adapt Heath-Brown's techniques to estimate these sums, from which
Theorem \ref{thm:density} follows.  In the appendix, we revisit our extension
problem from the point of view of Galois representations, and we discuss further
questions that arise.

% In Heath-Brown's work, the analytic formulation of the arithmetic problem
% follows from a classical $2$-descent argument; the innovation lies in
% the treatment of character sums with linked variables.  Similar character
% sums appear in this paper and are handled in the same way;
% our  main  task  is to formulate our inverse Galois problem
% in such a way that Heath-Brown's analytic machinery is applicable.

\subsection{Remarks on the literature}
After we completed this paper, we learned of the works of
%
% on
% the density of biquadratic fields that satisfy the Hasse norm principle, and
% and the work of
%
Fouvry--Luca--Pappalardi--Shparlinski \cite{fouvry}, 
%
% on counting dihedral and quaternionic extensions.
%
Rohrlich \cite{rohrlich-dih-type, rohrlich-quat, rohrlich-quat2, rohrlich-dihedral}, and
Rome \cite{rome},
on
similar enumerative questions about biquadratic fields.
%
% Both
% the main result and the approach of these papers are similar to ours; we
%
To understand the connections between this paper and the works of
  Rohrlich we need to first recast our results in the context of Galois
  representations; we will discuss that in the appendix.
  For the rest of this section we will comment on the works of
  Fouvry et al.~and of Rome.  We begin with the
work of Rome.

We say that a number field $K$ \textit{satisfies the Hasse norm principle}
(HNP) if for each rational
$
r\in\Q,
$
$r$ is the norm of an element in $K$ if and only if $r$ is the
norm of an element in every completion of
$K$.
The celebrated Hasse norm theorem \cite[p.~185]{cassels-frohlich} says
that cyclic extensions of $\Q$ satisfy the HNP. On the other hand, there are explicit examples of
biquadratic fields for which the HNP fails \cite[Exercises~5.3, 5.4]{cassels-frohlich}.
Rome \cite{rome} shows that there are
\begin{equation}
   \frac{1}{3\sqrt{2\pi}}
   \sqrt{X}(\log X)^{1/2}
   \prod_p
     \Big(\Bigl(
       1 - \frac{1}{p}
     \Bigr)^{3/2}
     \Bigl(
       1 + \frac{3}{2p}
     \Bigr)\Big)
  \;+\;
  O(\sqrt{X})            \label{rome}
\end{equation}
biquadratic fields of discriminant
$
\le X
$
for which the HNP fails. Recalling (\ref{px}), we see that Rome's result shows that almost all biquadratic
fields satisfy the HNP.

Rome expresses the \textit{failure} of the HNP for
$
\Q(\sqrt{d_1}, \sqrt{d_2})
$
in terms of quadratic symbols in a similar way to our criteria for the
\textit{existence} of $Q_8$- and $D_4$-extensions containing 
$
\Q(\sqrt{d_1}, \sqrt{d_2})
$,
cf.~\cite[Lemma 2.1]{rome} versus our Lemma \ref{lem:q8} and \ref{lem:d4}.
Thus it is not surprising that his estimate (\ref{rome}) is of the same
order of magnitude as ours---indeed the Euler products that appear in the main term are also the same. 

After converting his criteria into a counting function, Rome arrives at a
character sum similar to ours. To evaluate this sum he follows the approach of Friedlander--Iwaniec in \cite{iwaniec}, where they establish asymptotic formulae for the number of pairs $(a,b)$ up to given bounds such that the ternary quadratic form $ax^2+by^2-z^2$ has a nontrivial rational zero. We instead choose to follow the approach of Heath-Brown in \cite{heath}, which is easier to adapt to our particular situation. Any apparent differences in the approaches of Friedlander--Iwaniec and Heath-Brown are indeed superficial---they make different arrangments in how they handle the sums, but the main analytic inputs are the same: the Siegel--Walfisz theorem, estimates for bilinear forms involving the Jacobi symbol, and the Landau--Selberg--Delange method.

While Rome's criteria are similar to ours, there is a key difference:
the failure of the HNP depends only on the prime divisors of the $d_i$, but in our
setting the sign of the
$d_i$
also plays a role. This necessitates a parametrization of the biquadratic fields
that is similar to but more refined than the one in \cite{rome}, see~Section
\ref{sec:parametrize}.
Because of this and other particular complications in our setting (e.g.~Remark \ref{rem:d4-unique-cyclic}) we must work with many character sums, each corresponding to a given
Galois type and signature of the $d_i$. Much of our effort lies in wrangling
these different expressions to a ``canonical form'' so we only have to apply the
Heath-Brown machinery once; see~Propositions \ref{prop:fhq8} and \ref{prop:fhd4}.
In particular, there seems to be no simple way of repurposing Rome's estimates for needs, so we opt for doing our own arrangements and estimates.

Next
we turn to \cite{fouvry}.  Motivated by Malle's conjectures
\cite{malle1, malle2}, Fouvry et al. give asymptotic formulae for the
number of biquadratic extensions of
$\Q$ that can be extended to
$Q_8$-, $D_4$-
and respectively
$
C_4\times C_2
$-extensions; this is exactly the same problem we study in this paper.
However, instead of ordering the biquadratic fields by their
discriminants (as in Malle's conjectures, in Rome's work, and in this paper), in
\cite{fouvry} they work with the sets 
\begin{align*}
  \mathcal F(T)
  &:=
  \{ (a,b):  \text{$a, b$  distinct, square-free integers $1 < a, b \le T$} \},
  \\
  \mathcal F(T; G)
  &:=
  \Bigl\{
    (a,b)\in \mathcal F(T):
       \begin{array}{l}
         \text{$
               \Q(\sqrt{a}, \sqrt{b})
               $
               can be extended to a Galois}
         \\
         \text{extension over $\Q$ with Galois group $G$}
       \end{array}
  \Bigr\},
\end{align*}
and they show that
\begin{align*}
\label{fouvry-results}
   \#\mathcal F(T)
   =
   \frac{36}{\pi^4}T^2 + O(T^{3/2}),
         \qquad
   \#\mathcal F(T; G)
   =
   \Bigl(
     \frac{c(G)}
          {\pi^3}
    \prod_p
      \Bigl(
         1 + \frac{1}{2p(p+1)}
      \Bigr)
      +
      o(1)
   \Bigr)
   \frac{T^2}{\log^2 T},
\end{align*}
where $c(D_4) = 33$ and $c(Q_8) = 7$.  For
$
(a,b)\in\mathcal F(T)
$,
the field discriminant of
$
\Q(\sqrt{a}, \sqrt{b})
$
is
$
2^m (ab/\gcd(a,b))^2
$
for some $0\le m\le 6$ depending on $a,b$ modulo $4$. Thus it would be very difficult (likely impossible) to directly derive our results from the ones in \cite{fouvry} without effectively reproving them from scratch. Moreover, the set
$
\mathcal F(T)
$
does not cover
complex biquadratic fields, which (as we have mentioned above) introduce additional arithmetic complications.
In summary, the results in \cite{fouvry} are not equivalent or convertible
to ours.
We also  note that \cite{fouvry} makes use of the aforementioned methods of
Friedlander--Iwaniec \cite{iwaniec}.

The work of Fouvry et al. contains an additional result not considered in our
paper, namely they
determine an asymptotic formula for
$
\mathcal F(T; C_4\times C_2)
$.
A straightforward adaptation of our methods could be used to prove a $C_4\times C_2$-analogue of Theorem
\ref{thm:density} by using 
(\ref{sq}) in place of our criterion for $Q_8$- and $D_4$-extensions.
But as we will explain in the appendix, our interest in $Q_8$ and $D_4$-extensions of biquadratic fields lies in their connection with
irreducible Galois representations, and 
$
C_4\times C_2
$-extensions do not give rise to such representations. Therefore we do not pursue this result here.

\section{Parameterizing the biquadratic fields}
     \label{sec:parametrize}

Let $X$ be a large parameter going to infinity, and define
\[
\mathcal{B}(X)\;:=\;\{\text{biquadratic fields of discriminant $\le X$}\}.
                % \linelabel{beex}
\]
(Here and throughout we consider fields up to isomorphism; equivalently, we work in a fixed algebraic closure $\overline{\mathbb{Q}}$ of $\mathbb{Q}$.)
Our goal in this section is to parameterize this set of biquadratic fields in a
way that is amenable to the analytic treatment in later sections.

A biquadratic extension $K/\Q$ is determined by its three quadratic subfields, which are themselves determined by three distinct squarefree integers
$
d_1,d_2,d_3\ne 1
$.
Three such quadratic extensions
$
\Q( \sqrt{d_i})
$
come from a biquadratic extension precisely when
\[
d_3\;=\;  d_1d_2 \,/ \gcd(d_1,d_2)^2.
          % \linelabel{dee3}
\]
Therefore we put
\[
   \mathcal{S}
   \;=\;
   \{  (d_1,d_2,d_3):
         \text{$d_i \ne 1$ are distinct square-free integers with
           $
           d_3 = d_1 d_2 \, / \gcd(d_1,d_2)^2
           $}
   \},
\]
and we call the elements of
$\mathcal{S}$
\emph{admissible triples}.
%        \linelabel{admissibletriple}
The field discriminant
$
\Delta
$
of
$
\Q( \sqrt{d_1}, \sqrt{d_2}, \sqrt{d_3} )
$
is equal to the product of the discriminants of its three quadratic
subfields. Explicitly, by \cite[Exercise~2.42(f)]{marcus} we have
\begin{equation}
\Delta\;=\;
\begin{cases}
  d_1 d_2 d_3
  &
  \text{if every  $d_i\equiv \mymod{1}{4}$;}
  \\
  16 d_1 d_2 d_3
  &
  \text{if one $d_i$ is $\mymod{1}{4}$ and the other two are both 
    $
    \mymod{\text{$2$}}{4}
    $;}
  \\
  16 d_1 d_2 d_3
  &
  \text{if one $d_i$ is $\mymod{1}{4}$ and the other two are both 
    $
    \mymod{\text{$3$}}{4}
    $;}
  \\
  64 d_1 d_2 d_3
  &
  \text{if one $d_i$ is $\mymod{3}{4}$ and the other two are $\mymod{2}{4}$.}
\end{cases}
       \label{disc}
\end{equation}
In particular,
$
\Delta
$
is always positive.  The four cases above cover all possible congruence modulo
$4$
conditions on the $d_i$. For
$
1\le h\le 4
$,
we say that a biquadratic field corresponding to an admissible triple
$
(d_1, d_2, d_3)
$
is of ``type $h$'' if the integers $d_i$ satisfy the congruences in the
$h$-th case in \eqref{disc}, and we define
\[
\mathcal{B}_h(X)\;:=\;\{\text{biquadratic fields of discriminant $\le X$ and of type $h$}\}.
\]
Next we put
\begin{align*}
\mathcal{S}_1\;&:=\;\{(d_1,d_2,d_3)\in \mathcal{S}:d_1\equiv d_2\equiv d_3\equiv\mymod{1}{4}\}, \\
\mathcal{S}_2\;&:=\;\{(d_1,d_2,d_3)\in \mathcal{S}:d_1\equiv d_2\equiv\mymod{2}{4}, d_3\equiv\mymod{1}{4}\}, \\
\mathcal{S}_3\;&:=\;\{(d_1,d_2,d_3)\in \mathcal{S}:d_1\equiv d_2\equiv\mymod{3}{4}, d_3\equiv\mymod{1}{4}\}, \\
\mathcal{S}_4\;&:=\;\{(d_1,d_2,d_3)\in \mathcal{S}:d_1\equiv d_2\equiv\mymod{2}{4}, d_3\equiv\mymod{3}{4}\},
\end{align*}
and we set
\[
   \text{$\tau_1=1, \quad \tau_2=\tau_3=16$, and $\tau_4=64$.}
\]
With $1\le h \le 4$, define
\[
 \mathcal{S}_h(X)
    \;:=\;
    \{
      (d_1, d_2, d_3) \in \mathcal{S}_h:    \tau_hd_1 d_2 d_3 \le X
    \},
\]
and set
\[
      m_1=6,  \quad   m_2=m_3=m_4=2.
\]
The upshot of the above discussion is that we have defined an $m_h$-to-one
correspondence
%
%
% \footnote{{\tt \blu{Louis: You call this correspondence $K$;
%    I \underline{removed} this notation, because it could potential be
%    in conflict e.g.~with the next section, and (I think) we do not use the
%    \underline{name} of this correspondence after this section.}}}
%
%
\begin{align}
  \mathcal{S}_h(X)
  &
    \to
  \mathcal{B}_h(X)
                             \label{sx}
  \\
  (d_1,d_2,d_3)
  &
  \mapsto
  \Q(\sqrt{d_1}, \sqrt{d_2}, \sqrt{d_3} )
                              \nonumber
\end{align}
between sets of admissible triples and biquadratic fields. 

Next we wish to distinguish between the totally real and totally complex biquadratic fields. To each biquadratic field $K=\Q(\sqrt{d_1},\sqrt{d_2},\sqrt{d_3})$, we associate a pair of signs $\sigma=(\eps_1,\eps_2)$, where $\eps_i=\sign(d_i)$ for $i=1,2$. Note that $\sign(d_3)=\eps_1\eps_2$ is determined by the other two. Thus $K$ is totally real if $\sigma=(+1,+1)$, and $K$ is totally complex otherwise. Hence we put
\[
\mathcal{S}_h^\sigma(X)\;:=\;\{(d_1,d_2,d_3)\in \mathcal{S}_h(X):(\eps_1,\eps_2)=\sigma\}.
\]
Finally we set up notation for the extension problems. For $G=Q_8$ or $D_4$,
put
\begin{equation}
\label{char-fn}
\ind_G(d_1,d_2,d_3)\;:=\;
\Bigl\{
  \begin{array}{ll}
    1
    &
    \text{if $\Q(\sqrt{d_1}, \sqrt{d_2}, \sqrt{d_3} )$ admits a $G$-extension,}
    \\
    0
    &
    \text{otherwise}.
  \end{array}
\end{equation}
To count fields that admit specified extensions, we define
\begin{equation}
\label{bhxG}
B_h^\sigma(X;G)
\;:=\;
\frac{1}{m_h}\sum_{(d_1,d_2,d_3)\in \mathcal{S}_h^\sigma(X)}\ind_G(d_1,d_2,d_3)
\end{equation}
for $G=Q_8$ or $D_4$, and then we put
\begin{equation}
\label{sum-over-h}
{B}^\sigma(X;G)
\;:=\;
\sum_{1\le h\le 4}  {B}_h^\sigma(X;G).
\end{equation}

\section{Local conditions for $G$-extensions}
       \label{sec:criterion}

Here is the classical criterion for the 
$Q_8$-extension problem \cite[$\S 6.1$]{generic}.

\begin{thm}[Witt \cite{witt}]
    \label{thm:witt}
Let $k$ be a field of characteristic $\not=2$, and let 
$
K = k(\sqrt{a}, \sqrt{b})
$
be a biquadratic extension of $k$ with
$
a, b\in k
$.
Then $K$ is contained in a $Q_8$-extension of $k$ if and only if the quadratic
forms
$
aX^2 + bY^2 + abZ^2
$
and
$
U^2 + V^2 + W^2
$
are equivalent over $k$.

\end{thm}

For the $D_4$-extension problem we have the following
folklore
result \cite[p.~35]{generic}.

\begin{thm}
  \label{thm:dihedral}
Let
$
K=k(\sqrt{a}, \sqrt{b})
$
be a biquadratic extension of $k$.  Then $K/k$ is contained in a $D_4$-extension
$
F/k
$
such that
$
F/k(\sqrt{b})
$
is cyclic if and only if $ab$ is a norm in
$
k(\sqrt{a})/k
$.

\end{thm}

We now specialize to the case
$
k = \Q
$  
and express these criteria in terms of quadratic symbols. First we set some notation. 
For any prime $p$ and any non-zero integer $E$, define
\[
E_p\;:=\;m/p^\alpha\qquad\text{where }p^\alpha\;\|\;E,\; \text{with } \alpha\ge0.
\]
Thus $E_p=E$ if $p\nmid E$, and if $E$ is squarefree and
$
p\mid E
$,
then
$
E_p=E/p
$.
If a variable has a subscript, say $E=d_i$, we will write $E_p=d_{i,p}$.
In the lemma below and for the
rest of the paper, we write $(a/b)$ for the usual Kronecker symbol. We recall for future reference that for odd integers $n$, we have
\[
  \Big(\frac{2}{n}\Big)\;=\;\Big(\frac{2}{-n}\Big)\;=\;(-1)^{(n^2-1)/8}.
\] Lastly, we put
\begin{equation}
  \eta(a,b)
  \;:=\;
  (-1)^{\frac{a-1}{2}\frac{b-1}{2}}\;=\;\Big(\frac{-1}{a}\Big)^{\frac{b-1}{2}}.
%          \linelabel{eta}
\end{equation}
Note that if $a,b$ are odd, then $\eta(a,b+2)=(-1/a)\eta(a,b)$, and $\eta(a,a)=(-1/a)$. If $a,b$ are odd and positive, then by quadratic reciprocity we have
\[
\Big(\frac{a}{b}\Big)\;=\;\eta(a,b)\Big(\frac{b}{a}\Big).
\]
Denote by
$
\PP
$
the set of ordered pairs
$
(d_1, d_2)
$
of distinct, square-free integers $\not=1$.

\begin{lem}
  \label{lem:q8}
Let
$  
(d_1, d_2)\in\PP
$.
Then
        $
        \biquad
        $
is contained in a $Q_8$-extension if and only if $\sign(d_1,d_2)=(+,+)$ and
the following conditions hold for every prime number $p$:
  \[
    \begin{array}{r|rrrrrrrrrrrrrrrrrrrr}
      &
      p>2
      &&
      p=2
\\    \hline
      p\nmid d_1, p\nmid d_2
      &
      \text{no condition}
      &&
      \displaystyle
      \eta(d_1,d_2)\Big(\frac{-1}{d_1d_2}\Big)=1
\\
      p\,|\, d_1, p \nmid d_2
      &
      \displaystyle
      \Bigl( \frac{-d_2}{p} \Bigr) = 1  \rule{0pt}{17pt}
      &&
      \displaystyle
      \eta(d_{1,2},d_2)
      \Bigl( \frac{-1}{d_{1,2}} \Bigr)
      \Bigl( \frac{-2}{d_2} \Bigr)
      = 1      
\\
      p\nmid d_1, p\,|\,d_2
      &
      \displaystyle
      \Bigl( \frac{-d_1}{p} \Bigr) = 1 \rule{0pt}{19pt}
      &&
      \displaystyle
      \eta(d_1,d_{2,2})
      \Bigl( \frac{-2}{d_1} \Bigr)
      \Bigl( \frac{-1}{d_{2, 2}} \Bigr)
      = 1      
\\
      p\,|\,d_1, p\,|\,d_2
      &
      \displaystyle
      \Bigl( \frac{-d_{1,p} d_{2,p}}{p} \Bigr) = 1    \rule{0pt}{19pt}
      \rule{0pt}{15pt}
      &&
      \displaystyle
      \eta(d_{1,2},d_{2,2})
      \Bigl( \frac{-2}{d_{1, 2} d_{2,2}} \Bigr)
      = 1      
    \end{array}
  \]
\end{lem}

\begin{proof}
Both
  $
  d_i\not=0
  $,
so the two quadratic forms in Theorem \ref{thm:witt} have the same
rank as well as the same discriminant modulo
$
(\Q^\times)^2
$.
It then follows from \cite[Theorem 4.9]{serre:course} that these two forms are
$\Q$-equivalent if and only if they are equivalent at every completion of
$\Q$.
The latter is true, by \cite[Theorem 4.7]{serre:course}, if and only if the
following product of Hilbert symbols
  \begin{align*}
    (d_1, d_2)_v (d_1, d_1 d_2)_v (d_2, d_1 d_2)_v
  \end{align*}
is equal to
$1$
for every place $v$ of $\Q$.
Hilbert symbols are symmetric, multiplicatively bilinear, and take values in
$
\pm 1
$,
so this product simplifies to
  \begin{align}
    (d_1, d_2)_v (d_1, d_1)_v (d_2, d_2)_v.
\label{ab}    
  \end{align}
Since
$
(u, v)_\infty = -1
$
if and only if both $u, v$ are negative, the product (\ref{ab}) is equal to $1$
at
$
v=\infty
$
if and only if both
$
d_1, d_2
$
are positive.

Turning to the finite primes, let us first recall how to relate Hilbert
symbols to quadratic characters.
Let
$
a, b
$
be non-zero integers.  Given a prime number $p$, write
$
a = p^{\alpha_p} A_p, b = p^{\beta_p} B_p
$
with
$
p\nmid A_p B_p
$.
By \cite[Theorem 3.1]{serre:course}, the Hilbert symbol
$
(a, b)_p
$
is equal to
\begin{equation}
\renewcommand{\arraystretch}{2}
    (a, b)_p
    =
    \left\{
      \begin{array}{lllll}
\displaystyle
        \Bigl(
          \frac{-1}{p}
        \Bigr)^{\alpha_p {\beta_p}}
        \Bigl(
          \frac{A_p}{p}
        \Bigr)^{\beta_p}
        \Bigl(
          \frac{B_p}{p}
        \Bigr)^{\alpha_p}
    &&
        \text{if $p>2$;}
    \\
\displaystyle      
        \eta(A_2,B_2)
        \Bigl( \frac{2}{A_2} \Bigr)^{\beta_2}
        \Bigl( \frac{2}{B_2} \Bigr)^{\alpha_2}
    &&
        \text{if $p=2$.}
  \end{array}
  \right.
  \renewcommand{\arraystretch}{1}
\label{localsymbol}
\end{equation}

For an odd prime $p$, we can now readily compute the following Hilbert symbols,
from which the $p>2$ case of  Lemma \ref{lem:q8} follows:
\[
  \begin{array}{r|ccc}
    &
    (d_1, d_2)_p  &  (d_1, d_1)_p   &  (d_2, d_2)_p
    \\    \hline
    p\nmid d_1 d_2
    &
    1 & 1 & 1
       \rule{0pt}{12pt}    
    \\
    p \, | \, d_1, p \nmid  d_2
    &
    \displaystyle
    \Big( \frac{d_2}{p} \Big)
    &
    \displaystyle
    \Big( \frac{-1}{p} \Big)
    &
    1
       \rule{0pt}{12pt}    
    \\
    p \nmid  d_1, p \, | \, d_2
    &
    \displaystyle
    \Big( \frac{d_1}{p} \Big)
    &
    1
    &
    \displaystyle
    \Big( \frac{-1}{p} \Big)
       \rule{0pt}{12pt}    
    \\
    p \, | \, d_1, p \, | \,  d_2
    &
    \displaystyle
    \Big( \frac{-1}{p} \Big)
    \Big( \frac{d_{1, p}}{p} \Big)
    \Big( \frac{d_{2, p}}{p} \Big)
    &
    \displaystyle
    \Big( \frac{-1}{p} \Big)
    &
    \displaystyle
    \Big( \frac{-1}{p} \Big)
       \rule{0pt}{12pt}    
  \end{array}
\]
Finally,  we take $p=2$.  Again we apply (\ref{localsymbol}) to compute the
the following Hilbert symbols:
\[
  \begin{array}{r|ccc}
    &
    (d_1, d_2)_2  &  (d_1, d_1)_2   &  (d_2, d_2)_2
    \\    \hline
    2\nmid d_1 d_2
    &
    \eta(d_1,d_2)
    &
    \displaystyle
    \Big(\frac{-1}{d_1}\Big)
    &
    \displaystyle
    \Big(\frac{-1}{d_2}\Big)
       \rule{0pt}{15pt}    
    \\
    2 \, | \, d_1, 2 \nmid  d_2
    &
    \displaystyle
    \eta(d_{1,2},d_2)
    \Big(\frac{2}{d_2}\Big)
    &
    \displaystyle
    \Big(\frac{-1}{d_{1,2}}\Big)
    &
    \displaystyle
    \Big(\frac{-1}{d_2}\Big)
       \rule{0pt}{12pt}    
    \\
    2 \nmid  d_1, 2 \, | \, d_2
    &
    \displaystyle
    \eta(d_1,d_{2,2})
    \Big(\frac{2}{d_1}\Big)
    &
    \displaystyle
    \Big(\frac{-1}{d_1}\Big)
    &
    \displaystyle
    \Big(\frac{-1}{d_{2,2}}\Big)
       \rule{0pt}{12pt}    
    \\
    2 \, | \, d_1, 2 \, | \,  d_2
    &
    \displaystyle
    \eta(d_{1,2},d_{2,2})
    \Big(\frac{2}{d_{1,2}d_{2,2}}\Big)
    &
    \displaystyle
    \Big(\frac{-1}{d_{1,2}}\Big)
    &
    \displaystyle
    \Big(\frac{-1}{d_{2,2}}\Big)
       \rule{0pt}{12pt}    
  \end{array}
\]
Multiply these symbols together and the $p=2$ case of Lemma \ref{lem:q8}
follows.
\end{proof}

\begin{cor}
  \label{cor:b3q8}
Biquadratic fields in the set $\mathcal{B}_3(X)$ (see Section \ref{sec:parametrize}) do
not admit a
$Q_8$-extension.
\end{cor}

\begin{proof}
Given a pair
$
(d_1, d_2)\in\mathcal P
$,
to say that the biquadratic field
$
\Q(\sqrt{d_1}, \sqrt{d_2})
$
belongs to the set $B_3(X)$ is to say that 
$
d_1, d_2
$
are odd, square-free integers $\not=1$  such that
$
d_1 \equiv d_2 \equiv\mymod{3}{4}
$.
Then
$
\eta(d_1, d_2) (-1/d_1d_2) = -1
$,
and so   
$
\Q(\sqrt{d_1}, \sqrt{d_2})
$
fails the $2$-adic condition in Lemma \ref{lem:q8} for admitting a $Q_8$-extension.
\end{proof}

We now turn to the $D_4$-extension problem. Since $D_4$ has a unique $C_4$-subgroup, if $F/k$ is a $D_4$-extension then there is a unique intermediate quadratic subfield $K \subset M \subset F$ such that $F/M$ is cyclic. For $F/\Q$ with $\biquad\subset F$, there are three possibilities for the quadratic field $M/\Q$, which gives us the three cases in the following lemma.

\begin{lem}
  \label{lem:d4}
Let
$
(d_1, d_2)\in\PP
$.
Then
$
\biquad
$
is contained in a $D_4$-extension $F/\Q$ if and only if all conditions in any
one of the following three cases hold:

{\rm(i)}
$F/\Q$ is cyclic over $\Q(\sqrt{d_1})$ if $\sign(d_1,d_2)\ne(+1,-1)$ and the following conditions hold for every prime number $p$:
 \[
    \begin{array}{r|rrrrrrrrrrrrrrrrrrrr}
      &
      p>2
      &&
      p=2
\\    \hline
      p\nmid d_1, p\nmid d_2
      &
      \text{no condition}
      &&
      \displaystyle
      \eta(d_1+2,d_2) = 1
\\
      p\,|\, d_1, p \nmid d_2
      &
      \displaystyle
      \Bigl( \frac{d_2}{p} \Bigr) = 1  \rule{0pt}{17pt}
      &&
      \displaystyle
      \eta(d_{1,2}+2,d_2)
      \Bigl( \frac{2}{d_2} \Bigr)
      = 1      
\\
      p\nmid d_1, p\,|\,d_2
      &
      \displaystyle
      \Bigl( \frac{-d_1}{p} \Bigr) = 1 \rule{0pt}{19pt}
      &&
      \displaystyle
      \eta(d_1+2,d_{2,2})
      \Bigl( \frac{2}{d_1} \Bigr)
       = 1   
\\
      p\,|\,d_1, p\,|\,d_2
      &
      \displaystyle
      \Bigl( \frac{d_{1,p} d_{2,p}}{p} \Bigr) = 1    \rule{0pt}{19pt}
      \rule{0pt}{15pt}
      &&
      \displaystyle
      \eta(d_{1,2}+2,d_{2,2})
      \Bigl( \frac{2}{d_{1,2} d_{2,2} } \Bigr)
       = 1     
    \end{array}
  \]

{\rm(ii)}
  $F/\Q$ is cyclic over $\Q(\sqrt{d_2})$ if $\sign(d_1,d_2)\ne(-1,+1)$ and the
  following conditions hold for every prime number $p$:
 \[
    \begin{array}{r|rrrrrrrrrrrrrrrrrrrr}
      &
      p>2
      &&
      p=2
\\    \hline
      p\nmid d_1, p\nmid d_2
      &
      \text{no condition}
      &&
      \displaystyle
      \eta(d_1,d_2+2) = 1
\\
      p\,|\, d_1, p \nmid d_2
      &
      \displaystyle
      \Bigl( \frac{-d_2}{p} \Bigr) = 1  \rule{0pt}{17pt}
      &&
      \displaystyle
      \eta(d_{1,2},d_2+2)
      \Bigl( \frac{2}{d_2} \Bigr)
      = 1      
\\
      p\nmid d_1, p\,|\,d_2
      &
      \displaystyle
      \Bigl( \frac{d_1}{p} \Bigr) = 1 \rule{0pt}{19pt}
      &&
      \displaystyle
      \eta(d_1,d_{2,2}+2)
      \Bigl( \frac{2}{d_1} \Bigr)
       = 1   
\\
      p\,|\,d_1, p\,|\,d_2
      &
      \displaystyle
      \Bigl( \frac{d_{1,p} d_{2,p}}{p} \Bigr) = 1    \rule{0pt}{19pt}
      \rule{0pt}{15pt}
      &&
      \displaystyle
      \eta(d_{1,2},d_{2,2}+2)
      \Bigl( \frac{2}{d_{1,2} d_{2,2} } \Bigr)
       = 1     
    \end{array}
  \]

{\rm(iii)}
  $F/\Q$ is cyclic over $\Q(\sqrt{d_1d_2})$ if $\sign(d_1,d_2)\ne(-1,-1)$ and the
  following conditions hold for every prime number $p$:
 \[
    \begin{array}{r|rrrrrrrrrrrrrrrrrrrr}
      &
      p>2
      &&
      p=2
\\    \hline
      p\nmid d_1, p\nmid d_2
      &
      \text{no condition}
      &&
      \displaystyle
      \eta(d_1,d_2) = 1
\\
      p\,|\, d_1, p \nmid d_2
      &
      \displaystyle
      \Bigl( \frac{d_2}{p} \Bigr) = 1  \rule{0pt}{17pt}
      &&
      \displaystyle
      \eta(d_{1,2},d_2)
      \Bigl( \frac{2}{d_2} \Bigr)
      = 1      
\\
      p\nmid d_1, p\,|\,d_2
      &
      \displaystyle
      \Bigl( \frac{d_1}{p} \Bigr) = 1 \rule{0pt}{19pt}
      &&
      \displaystyle
      \eta(d_1,d_{2,2})
      \Bigl( \frac{2}{d_1} \Bigr)
       = 1   
\\
      p\,|\,d_1, p\,|\,d_2
      &
      \displaystyle
      \Bigl( \frac{-d_{1,p} d_{2,p}}{p} \Bigr) = 1    \rule{0pt}{19pt}
      \rule{0pt}{15pt}
      &&
      \displaystyle
      \eta(d_{1,2},d_{2,2})
      \Bigl( \frac{2}{d_{1,2} d_{2,2} } \Bigr)
       = 1     
    \end{array}
  \]
\end{lem}

\begin{proof}
By the Hasse norm theorem, a rational integer
$
d_1 d_2
$
is a norm in $\Q(\sqrt{d_1})/\Q$ if and only if it is a norm in the extension of
local fields
$
\Q_v(\sqrt{d_1})/\Q_v
$
for every place $v$ of $\Q$.  Recalling Theorem \ref{thm:dihedral}, we see that
$
\Q(\sqrt{d_1}, \sqrt{d_2})
$
is contained in a $D_4$-extension that is cyclic over
$
\Q(\sqrt{d_2})
$
if and only if
\begin{equation}
  \text{$d_1 d_2 = U^2 - d_1 W^2$
    has
    a solution in every completion $\Q_v$ of $\Q$.}         \label{globald4}
\end{equation}
Since
$
d_i\not=0
$,
necessarily $U, W$ cannot be both zero, and so (\ref{globald4}) implies that
\begin{equation}
  \text{$d_1 x^2 + d_1 d_2 y^2 = z^2$  has a  solution in
    $
    ( \Q_v \smallsetminus \{ 0 \} )^3
    $
    for each place $v$ of $\Q$.}          \label{localone}
\end{equation}
By definition of Hilbert symbols, (\ref{localone}) is equivalent to
\begin{equation}
  \text{$(d_1, d_1 d_2)_v = 1$  for all places $v$ of $\Q$.}
        \label{locald4}
\end{equation}
Note that (\ref{localone}), and hence (\ref{locald4}), \textit{is}
equivalent
to (\ref{globald4}):  For each place $v$ of $\Q$, let
$
(a,b,c)\in (\Q_v \smallsetminus \{0\})^3
$
be a solution to 
$
d_1 x^2 + d_1 d_2 y^2 = z^2
$.
If $b\not=0$ then
$
U = c/b, W = a/b
$
is a $\Q_v$-solution to
$
d_1 d_2 = U^2 - d_1 W^2
$.
And if
$
b=0
$
for a given place $v$, then
$
d_1 = (c/a)^2,
$
in which case
\[
   U =  \frac{d_2 + d_1}{2},
   \quad
   W =  \frac{d_2 - d_1}{2c/a}
\]
is a $\Q_v$-solution to (\ref{globald4}).

To recapitulate,
  $
  \Q(\sqrt{d_1}, \sqrt{d_2})
  $
is contained in a $D_4$-extension that is cyclic over
  $
  \Q(\sqrt{d_2})
  $
if and only if (\ref{locald4}) holds.  We then evaluate the Hilbert symbols in
       (\ref{locald4})
using
the two tables in the proof of Lemma
\ref{lem:q8} and we get the
second table in Lemma \ref{lem:d4}, recalling that for odd integers
  $
  a, b
  $,
we have
\[  
   \eta(a, b) \Bigl( \frac{-1}{a} \Bigr) = \eta(a, b+2).
\]
When
$
\Q(\sqrt{d_1}, \sqrt{d_2})
$
is contained in a $D_4$-extension that is cyclic over
$
\Q(\sqrt{d_2})
$
or
$
\Q(\sqrt{d_1 d_2})
$
instead, the corresponding Hilbert symbols conditions are
  \begin{equation*}
    \begin{array}{rrrlllllllllllll}
      \text{cyclic over
        $
        \Q(\sqrt{d_1})
        $:}
      &&
    (d_2, d_1 d_2)_v = 1
    &
    \text{for all places $v$ of $\Q$;}
    \\
      \text{cyclic over
        $
        \Q(\sqrt{d_1 d_2})
        $:}
      &&
    (d_1, d_2)_v = 1
    &
    \text{for all places $v$ of $\Q$.}
  \end{array}
  \end{equation*}
Then these cases follow similarly to the above.
\end{proof}

\begin{remm}
  \label{rem:d4-unique-cyclic}
If $F/\Q$ is a $D_4$-extension containing
$
\Q(\sqrt{d_1}, \sqrt{d_2})
$,
then $F$ is cyclic over exactly one of its quadratic subfields.  However, if
$
\Q(\sqrt{d_1}, \sqrt{d_2})
$
admits a $D_4$-extension, then it is possible that it can be
extended into different $D_4$-extensions
$
F_1, F_2
$
over $\Q$ such that each $F_i$ is cyclic over different quadratic subfields of
$
\Q(\sqrt{d_1}, \sqrt{d_2})
$.
This introduces some complications in the proof of Theorem \ref{thm:density}; see \eqref{d4-to-mi}, for instance.
\end{remm}

\section{Characteristic functions for $G$-extensions}
    \label{sec:char}

Recall our definition of admissible triples $(d_1,d_2,d_3)$ from Section \ref{sec:parametrize}, and recall our definition of
$
{\boldsymbol 1}_{G}(d_1, d_2, d_3),
$
the characteristic function for when the biquadratic field
$
        \Q(\sqrt{d_1}, \sqrt{d_2}, \sqrt{d_3} )
$
admits a $G$-extension (for $G=Q_8$ or $D_4$). In this section we develop explicit formulas for $\ind_G(d_1,d_2,d_3)$ that are divisor sums over quadratic characters.

First we introduce some notation. For
    $
    1\le i\le3$, $1\le j\le2
    $,
the six variables
$
D_{ij}
$
will always denote positive, odd, squarefree integers. In every case below, we
will always have
$
D_i=D_{i1}D_{i2}
$
for each
$
1\le i\le 3
$
(even though the $D_i$ will be defined differently in terms of the $d_i$ in different cases). When referring to these variables together in a tuple, we write
\[
\vec{D}\;:=\;(D_{11},D_{12},D_{21},D_{22},D_{31},D_{32}).
\]
We will use the notation
\begin{equation}
\label{omega-of-d}
\omega(\vec{D})\;:=\;\omega(D_{11}D_{12}D_{21}D_{22}D_{31}D_{32}),
\end{equation}
where $\omega(n)$ denotes the number of distinct prime divisors of $n$, and we define the function
\begin{equation}
\label{g-of-d}
g(\vec{D})\;:=\;\Big(\frac{D_{22}D_{32}}{D_{11}}\Big)\Big(\frac{D_{12}D_{32}}{D_{21}}\Big)\Big(\frac{D_{12}D_{22}}{D_{31}}\Big).
\end{equation}
We will use the notation
$
\sum_{D_i=D_{i1}D_{i2}}
$
to denote a sum running over all positive divisors
$
D_{i1}\mid D_i
$
for each $i=1,2,3$ (again, note that $D_{i2}$ is the complementary divisor, $D_i=D_{i1}D_{i2}$). In every case below, we will write $\eps_i=\sign(d_i)$ for $i=1,2$. For notational convenience, for odd integers $a, b, c$ we define the function
\[
      \Phi(a,b,c)
      \;:=\;
      \eta(a,b+2)
      \eta(b,c+2)
      \eta(c,a+2),
\]
and for any non-zero integer $n$, we define
\[
\psi(n)\;:=\;
\Bigl\{
\begin{array}{llllll}
  1 & \text{if every prime divisor of $n$ is $\equiv\mymod{1}{4}$,}
  \\
  0 & \text{otherwise}.
\end{array}
\]

\subsection{$Q_8$-extension}

%% Should we maybe be writing d_1 = D_2D_3 instead, etc.? More symmetrical maybe...

\begin{lem}
\label{lem:q8-embedding-function}
Let
$
(d_1,d_2,d_3)
$
be an admissible triple as defined in Section \ref{sec:parametrize}.  If
$
(\eps_1,\eps_2)\ne(1,1)
$,
then
$
\ind_{Q_8}(d_1,d_2,d_3)=0
$.
Otherwise
$
d_1,d_2>0
$,
and we have two cases:
\begin{itemize}
\item
  If
  $
  2\nmid d_1 d_2
  $,
  put $D_3=\gcd(d_1,d_2)>0$,
  $
  D_1 = d_1/D_3
  $,
  and
  $
  D_2 = d_2/D_3
  $.
  In this case,
\[
\ind_{Q_8}(d_1,d_2,d_3)\;=\;\ind_{Q_8}(D_1D_3,D_2D_3,D_1D_2)\;=\sum_{D_i=D_{i1}D_{i2}} F_{\odd}(\vec{D};Q_8)2^{-\omega(\vec{D})}g(\vec{D}),
\]
where we have put
\[
F_{\odd}(\vec{D};Q_8)\;:=\;
\frac{1}
       {2}
  \Big(
      1
      +
      \eta(D_{1}D_{3},D_2D_3)
      \Big(
         \frac{-1}
              {D_1D_2}
      \Big)
  \Big)
       \Phi(D_{11},D_{21},D_{31}).
\]

\item
  If
  $
  2\mid \gcd(d_1,d_2)
  $,
  put
  $
  D_3=\gcd(d_1,d_2)/2>0
  $,
  $
  D_1 = d_1/(2D_3)
  $,
  and
  $
  D_2 = d_2/(2D_3)
  $.
In this case we have
\[
\ind_{Q_8}(d_1,d_2,d_3)\;=\;\ind_{Q_8}(2D_1D_3,2D_2D_3,D_1D_2)
\;=\sum_{D_i=D_{i1}D_{i2}}F_{\even}(\vec{D};Q_8)2^{-\omega(\vec{D})}g(\vec{D}),
\]
where
\[
F_{\even}(\vec{D};Q_8)\;:=\;\frac{1}{2}\Big(1+\eta(D_{1}D_{3},D_2D_3)\Big(\frac{-2}{D_1D_2}\Big)\Big)\Big(\frac{2}{D_{11}D_{21}}\Big)\Phi(D_{11},D_{21},D_{31}).
\]
\end{itemize}
\end{lem}

\begin{remm}
Note that one can have
$
2\mid d_1
$
and
$
2\nmid d_2
$
(or vice versa) in an admissible triple. However, in our parameterization
  of 
biquadratic fields via the sets
$\mathcal{S}_h$,
we have arranged for $d_1,d_2$ to always have the same parity, so we need not consider these other cases.
\end{remm}

\begin{proof}
Lemma \ref{lem:q8} shows that
\[
\ind_{Q_8}(d_1,d_2,d_3)=0\qquad\text{when }\sign(d_1,d_2)\ne(1,1).
\]
So from now on suppose $d_1,d_2>0$. First, further suppose that
$
2 \nmid d_1 d_2
$.
We take a product over all of the local conditions in the table in Lemma
\ref{lem:q8}, which gives us
\begin{align*}
  \ind_{Q_8}(d_1,d_2,d_3)
  \;=\;
  \frac{1}{2}\Big(1+&\eta(d_1,d_2)\Big(\frac{-1}{d_1d_2}\Big)\Big)
  \\
  &
  \times \prod_{\substack{p\mid d_1 \\ p\nmid d_2}}
  \frac{1}
       {2}
  \Big(
       1 + \Big( \frac{-d_2}{p} \Big)
  \Big)
  \prod_{\substack{p\nmid d_1 \\ p\mid d_2}}
  \frac{1}
       {2}
  \Big(
        1 + \Big( \frac{-d_1}{p} \Big)
  \Big)
  \prod_{\substack{p\mid d_1 \\ p\mid d_2}}
  \frac{1}
       {2}
  \Big(
       1 + \Big( \frac{-d_{1,p}d_{2,p}}{p} \Big)
  \Big).
\nonumber
\end{align*}
Writing
$
D_3=(d_1,d_2)>0
$,
$
D_1 = d_1/D_3,
$
$
D_2 = d_2/D_3
$,
the triple product above is equal to
  \begin{align}
    \frac{1}
         {2^{\omega(D_1)+\omega(D_2)+\omega(D_3)}}
         &
    \underset{{p | D_1}}{\prod}
\Big(1+\Big(\frac{-D_2D_3}{p}\Big)\Big)
  \underset{{p | D_2}}{\prod}
\Big(1+\Big(\frac{-D_1D_3}{p}\Big)\Big) 
  \underset{{p | D_3}}{\prod}   
    \Bigl(
        1 + \Big( \frac{-D_1\frac{D_3}{p} D_2\frac{D_3}{p}}{p} \Big)
    \Bigr)        
\nonumber
\\    
    &=
    \:
    2^{-\omega(D_1 D_2 D_3)}
    \sum_{D_{11} | D_1}
      \Bigl( \frac{-D_2 D_3}{D_{11}} \Bigl)
    \sum_{D_{21} | D_2}
      \Bigl( \frac{-D_1 D_3}{D_{21}} \Bigl)
    \sum_{D_{31} | D_3}
      \Bigl( \frac{-D_1 D_2}{D_{31}} \Bigl).
\label{original}
  \end{align}
Writing $D_i=D_{i1}D_{i2}$ for $i=1,2,3$, the above expression becomes
  \begin{align*}
    &
    2^{-\omega(D_1 D_2 D_3)}    \sum_{D_{i1}\mid D_i}
      \Bigl(
        \frac{-D_{21} D_{22} D_{31} D_{32} }{ D_{11} }
      \Bigl)
      \Bigl(
        \frac{-D_{11} D_{12} D_{31} D_{32} }{ D_{21} }
      \Bigl)
      \Bigl(
        \frac{-D_{11} D_{12} D_{21} D_{22} }{ D_{31} }
      \Bigl),
  \end{align*}
where the sum runs through all ordered triples
$
(D_{11}, D_{21}, D_{31})
$
of positive divisors of
$
(D_1, D_2, D_3)
$.
Since the $D_{ij}$ are odd, we can apply quadratic reciprocity to rewrite the
summand as
  \begin{align*}
      \eta(D_{11},D_{21}+2)
      &\eta(D_{21},D_{31}+2)
      \eta(D_{31},D_{11}+2)
      \Bigl(
        \frac{ D_{22} D_{32} }{ D_{11} }
      \Bigl)
      \Bigl(
        \frac{ D_{12} D_{32} }{ D_{21} }
      \Bigl)
      \Bigl(
        \frac{ D_{12} D_{22} }{ D_{31} }
      \Bigl) \\
      &=\;\Phi(D_{11},D_{21},D_{31})g(\vec{D}).
  \end{align*}
Putting everything together completes the proof in the case
$
2\nmid d_1 d_2
$.

Next, suppose
  $
  2 | \gcd(d_1, d_2)
  $.
  Proceeding as above, we get
\begin{align*}
\ind_{Q_8}(d_1,d_2,d_3)\;=\;\frac{1}{2}\Big(1+&\eta(d_{1,2},d_{2,2})\Bigl( \frac{-2}{d_{1, 2} d_{2,2}} \Bigr)\Big) \\
&\times\prod_{\substack{p\mid d_1 \\ p\nmid d_2 \\ p>2}}\frac{1}{2}\Big(1+\Big(\frac{-d_2}{p}\Big)\Big)\prod_{\substack{p\nmid d_1 \\ p\mid d_2 \\p>2}}\frac{1}{2}\Big(1+\Big(\frac{-d_1}{p}\Big)\Big)\prod_{\substack{p\mid d_1 \\ p\mid d_2 \\ p>2}}\frac{1}{2}\Big(1+\Big(\frac{-d_{1,p}d_{2,p}}{p}\Big)\Big)
\end{align*}
using the results of Lemma \ref{lem:q8}. In this case we define
$
D_1, D_2, D_3
$
differently:  we set
$
D_3 = \gcd(d_1, d_2)/2
$,
$
D_1 = d_1/(2D_3),
$,
and $D_2 = d_2/(2D_3)$.
Then the triple product is equal to
  \begin{align}
         2^{-\omega(D_1D_2D_3)}
         &
    \underset{{p | D_1}}{\prod}
\Big(1+\Big(\frac{-2D_2D_3}{p}\Big)\Big)
  \underset{{p | D_2}}{\prod}
\Big(1+\Big(\frac{-2D_1D_3}{p}\Big)\Big) 
  \underset{{p | D_3}}{\prod}   
    \Bigl(
        1 + \Big( \frac{-D_1\frac{2D_3}{p} D_2\frac{2D_3}{p}}{p} \Big)
    \Bigr)        
\nonumber
\\    
    &=
    \:
    2^{-\omega(D_1 D_2 D_3)}
    \sum_{D_{11} | D_1}
      \Bigl( \frac{-2D_2 D_3}{D_{11}} \Bigl)
    \sum_{D_{21} | D_2}
      \Bigl( \frac{-2D_1 D_3}{D_{21}} \Bigl)
    \sum_{D_{31} | D_3}
      \Bigl( \frac{-D_1 D_2}{D_{31}} \Bigl)
\nonumber
\\
&=\;2^{-\omega(D_1 D_2 D_3)}\sum_{D_{11},D_{21},D_{31}}\Big(\frac{2}{D_{11}D_{21}}\Big)\Phi(D_{11},D_{21},D_{31})\Big(\frac{D_{22}D_{32}}{D_{11}}\Big)\Big(\frac{D_{12}D_{32}}{D_{21}}\Big)\Big(\frac{D_{12}D_{22}}{D_{31}}\Big),
\end{align}
  where again we have put $D_i=D_{i1}D_{i2}$ for each $i$. This completes the
  proof of
  Lemma \ref{lem:q8-embedding-function}.
\end{proof}

\subsection{$D_4$-extension}
In
  this section we prove an analog of Lemma \ref{lem:q8-embedding-function} for
$
D_4
$-extensions.  The basic idea is the same as for the $Q_8$ case, but with an
important complication:  we noted in Remark \ref{rem:d4-unique-cyclic} that
a biquadratic field can be extended to a $D_4$-extension up to three
different ways, but each extendable biquadratic field should be counted only once.  Thus we have an
analytic expression for each of the
$
2^3-1
$
possible extensions, and then we apply an inclusion-exclusion argument to
arrive at the final expression. These considerations are what are responsible for the more complicated expressions in the following lemma.

\begin{lem}
\label{lem:d4-embedding-function}
Let $(d_1,d_2,d_3)$ be an admissible triple; we have two cases:
\begin{itemize}
\item If $2\nmid d_1d_2$, put $D_3=\gcd(d_1,d_2)>0$, $D_1=d_1/\eps_1D_3$, and $D_2=d_2/\eps_2D_3$. Then
\[
\ind_{D_4}(d_1,d_2,d_3)\;=\;\ind_{D_4}(\eps_1D_1D_3,\eps_2D_3D_3,\eps_1\eps_2D_1D_2)\;=\sum_{D_i=D_{i1}D_{i2}} F_{\odd}(\vec{D};D_4)2^{-\omega(\vec{D})}g(\vec{D}),
\]
where we have put
\begin{align*}
F_{\odd}(\vec{D};D_4)\;:&=\;\bigg\{\ind_{(\eps_1,\eps_2)\ne(+1,-1)}\cdot\frac{1}{2}\Big(1+\eta(\eps_1D_1D_3+2,\eps_2D_2D_3)\Big)\Big(\frac{-1}{D_{11}D_{31}}\Big) \\
&\phantom{blahblah}+\ind_{(\eps_1,\eps_2)\ne(-1,+1)}\cdot\frac{1}{2}\Big(1+\eta(\eps_1D_1D_3,\eps_2D_2D_3+2)\Big)\Big(\frac{-1}{D_{21}D_{31}}\Big) \\
&\phantom{blahblah}+\ind_{(\eps_1,\eps_2)\ne(-1,-1)}\cdot\frac{1}{2}\Big(1+\eta(\eps_1D_1D_3,\eps_2D_2D_3)\Big)\Big(\frac{-1}{D_{11}D_{21}}\Big)\bigg\} \\
&\phantom{blahblahblahblah}\times\Big(\frac{\eps_1}{D_{21}D_{31}}\Big)\Big(\frac{\eps_2}{D_{11}D_{31}}\Big)\Phi(D_{11},D_{21},D_{31}) \\
&\phantom{blah}+\bigg\{-\ind_{(\eps_1,\eps_2)\ne(\pm1,\mp1)}\psi(D_1D_2)-\ind_{(\eps_1,\eps_2)\ne(\pm1,-1)}\psi(D_2D_3) \\
&\phantom{blahblah}-\ind_{(\eps_1,\eps_2)\ne(-1,\pm1)}\psi(D_1D_3)+\ind_{(\eps_1,\eps_2)=(+1,+1)}\psi(D_1D_2D_3)\bigg\}.
\end{align*}
\item If $2\mid \gcd(d_1,d_2)$, put $D_3=\gcd(d_1,d_2)/2>0$, $D_1=d_1/\eps_12D_3$, $D_2=d_2/\eps_22D_3$. Then
\[
\ind_{D_4}(d_1,d_2,d_3)\;=\;\ind_{D_4}(\eps_12D_1D_3,\eps_22D_3D_3,\eps_1\eps_2D_1D_2)\;=\sum_{D_i=D_{i1}D_{i2}} F_{\even}(\vec{D};D_4)2^{-\omega(\vec{D})}g(\vec{D}),
\]
where we have put
\begin{align*}
F_{\even}(\vec{D};D_4)\;:&=\;\bigg\{\ind_{(\eps_1,\eps_2)\ne(+1,-1)}\cdot\frac{1}{2}\Big(1+\eta(\eps_1D_1D_3+2,\eps_2D_2D_3)\Big(\frac{2}{D_1D_2}\Big)\Big)\Big(\frac{-1}{D_{11}D_{31}}\Big) \\
&\phantom{blahblah}+\ind_{(\eps_1,\eps_2)\ne(-1,+1)}\cdot\frac{1}{2}\Big(1+\eta(\eps_1D_1D_3,\eps_2D_2D_3+2)\Big(\frac{2}{D_1D_2}\Big)\Big)\Big(\frac{-1}{D_{21}D_{31}}\Big) \\
&\phantom{blahblah}+\ind_{(\eps_1,\eps_2)\ne(-1,-1)}\cdot\frac{1}{2}\Big(1+\eta(\eps_1D_1D_3,\eps_2D_2D_3)\Big(\frac{2}{D_1D_2}\Big)\Big)\Big(\frac{-1}{D_{11}D_{21}}\Big)\bigg\} \\
&\phantom{blahblahblahblah}\times\Big(\frac{\eps_1}{D_{21}D_{31}}\Big)\Big(\frac{\eps_2}{D_{11}D_{31}}\Big)\Big(\frac{2}{D_{11}D_{21}}\Big)\Phi(D_{11},D_{21},D_{31}) \\
&\phantom{blah}+\bigg\{-\ind_{(\eps_1,\eps_2)\ne(\pm1,\mp1)}\psi(D_1D_2)-\ind_{(\eps_1,\eps_2)\ne(\pm1,-1)}\psi(D_2D_3) \\
&\phantom{blahblah}-\ind_{(\eps_1,\eps_2)\ne(-1,\pm1)}\psi(D_1D_3)+\ind_{(\eps_1,\eps_2)=(+1,+1)}\psi(D_1D_2D_3)\bigg\} \\
&\phantom{blahblahblahblah}\times\frac{1}{2}\Big(1+\Big(\frac{2}{D_1D_2}\Big)\Big).
\end{align*}
\end{itemize}
\end{lem}

\begin{proof}
We proceed as in the proof of Lemma \ref{lem:q8-embedding-function}, though now we must separately consider the cases where $\Q(\sqrt{d_1},\sqrt{d_2},\sqrt{d_3})$ is cyclic over $M_1=\Q(\sqrt{d_1})$, $M_2=\Q(\sqrt{d_2})$, and $M_3=\Q(\sqrt{d_3})=\Q(\sqrt{d_1d_2})$, and then we combine them. Specifically, we define
\[
\ind_{M_i}(d_1,d_2,d_3)\;:=\;
\begin{cases}
1 & \text{if $\Q(\sqrt{d_1},\sqrt{d_2},\sqrt{d_3})$ admits a $D_4$-extension that is cyclic over $M_i$,} \\
0 & \text{otherwise},
\end{cases}
\]
and we write $\ind_{M_1}\ind_{M_2}$, etc. for products of the indicator functions. By inclusion-exclusion, we have the identity
\begin{equation}
  \ind_{D_4}
  \;=\;
  \ind_{M_1}
  +
  \ind_{M_2}
  +
  \ind_{M_3}
  -
  \ind_{M_1}\ind_{M_2}
  -
  \ind_{M_1}\ind_{M_3}
  -
  \ind_{M_2}\ind_{M_3}
  +
  \ind_{M_1}\ind_{M_2}\ind_{M_3}
                         \label{d4-to-mi}
\end{equation}
of functions of admissible triples.
This identity holds because each $D_4$-extension is cyclic over exactly one of the $M_i$, though a given biquadratic field $K$ may admit different $D_4$-extensions that are cyclic over different quadratic subfields $M_i$; see Remark \ref{rem:d4-unique-cyclic}.

First we assume $2\nmid d_1d_2$. In each case of being cyclic over a given $M_i$, we take a product over the local conditions in the relevant table in Lemma \ref{lem:d4}, which gives
\begin{align*}
\ind_{M_1}(d_1,d_2,d_3)\;&=\;\ind_{\sign(d_1,d_2)\ne(+1,-1)}\cdot\frac{1}{2}\Big(1+\eta(d_1+2,d_2)\Big) \\
&\phantom{blahblah}\times \prod_{\substack{p\mid d_1 \\ p\nmid d_2}}\frac{1}{2}\Big(1+\Big(\frac{d_2}{p}\Big)\Big)\prod_{\substack{p\nmid d_1 \\ p\mid d_2}}\frac{1}{2}\Big(1+\Big(\frac{-d_1}{p}\Big)\Big)\prod_{\substack{p\mid d_1 \\ p\mid d_2}}\frac{1}{2}\Big(1+\Big(\frac{d_{1,p}d_{2,p}}{p}\Big)\Big), \\
\ind_{M_2}(d_1,d_2,d_3)\;&=\;\ind_{\sign(d_1,d_2)\ne(-1,+1)}\cdot\frac{1}{2}\Big(1+\eta(d_1,d_2+2)\Big) \\
&\phantom{blahblah}\times \prod_{\substack{p\mid d_1 \\ p\nmid d_2}}\frac{1}{2}\Big(1+\Big(\frac{-d_2}{p}\Big)\Big)\prod_{\substack{p\nmid d_1 \\ p\mid d_2}}\frac{1}{2}\Big(1+\Big(\frac{d_1}{p}\Big)\Big)\prod_{\substack{p\mid d_1 \\ p\mid d_2}}\frac{1}{2}\Big(1+\Big(\frac{d_{1,p}d_{2,p}}{p}\Big)\Big),
\end{align*}
\begin{align*}
\ind_{M_3}(d_1,d_2,d_3)\;&=\;\ind_{\sign(d_1,d_2)\ne(-1,-1)}\cdot\frac{1}{2}\Big(1+\eta(d_1,d_2)\Big) \\
&\phantom{blahblah}\times \prod_{\substack{p\mid d_1 \\ p\nmid d_2}}\frac{1}{2}\Big(1+\Big(\frac{d_2}{p}\Big)\Big)\prod_{\substack{p\nmid d_1 \\ p\mid d_2}}\frac{1}{2}\Big(1+\Big(\frac{d_1}{p}\Big)\Big)\prod_{\substack{p\mid d_1 \\ p\mid d_2}}\frac{1}{2}\Big(1+\Big(\frac{-d_{1,p}d_{2,p}}{p}\Big)\Big).
\end{align*}
In each case put $d_1=\eps_1D_1D_3$ and $d_2=\eps_2D_2D_3$, where $D_3=\gcd(d_1,d_2)>0$ and $\eps_i=\sign(d_i)$, so the $D_i$ are odd, positive, squarefree, and pairwise coprime. We follow the same steps as in the proof of Lemma \ref{lem:q8-embedding-function} to express each of the above products as a sum, which gives
\begin{align*}
\ind_{M_1}(\eps_1D_1D_3,&\eps_2D_2D_3,\eps_1\eps_2D_1D_2)\;=\;\ind_{(\eps_1,\eps_2)\ne(+1,-1)}\cdot\frac{1}{2}\Big(1+\eta(\eps_1D_1D_3+2,\eps_2D_2D_3)\Big) \\
&\phantom{blahblah}\times\sum_{D_i=D_{i1}D_{i2}}\Big(\frac{\eps_1}{D_{21}D_{31}}\Big)\Big(\frac{\eps_2}{D_{11}D_{31}}\Big)\Big(\frac{-1}{D_{11}D_{31}}\Big)\Phi(D_{11},D_{21},D_{31})2^{-\omega(\vec{D})}g(\vec{D}), \\
\ind_{M_2}(\eps_1D_1D_3,&\eps_2D_2D_3,\eps_1\eps_2D_1D_2)\;=\;\ind_{(\eps_1,\eps_2)\ne(-1,+1)}\cdot\frac{1}{2}\Big(1+\eta(\eps_1D_1D_3,\eps_2D_2D_3+2)\Big) \\
&\phantom{blahblah}\times\sum_{D_i=D_{i1}D_{i2}}\Big(\frac{\eps_1}{D_{21}D_{31}}\Big)\Big(\frac{\eps_2}{D_{11}D_{31}}\Big)\Big(\frac{-1}{D_{21}D_{31}}\Big)\Phi(D_{11},D_{21},D_{31})2^{-\omega(\vec{D})}g(\vec{D}), \\
\ind_{M_3}(\eps_1D_1D_3,&\eps_2D_2D_3,\eps_1\eps_2D_1D_2)\;=\;\ind_{(\eps_1,\eps_2)\ne(-1,-1)}\cdot\frac{1}{2}\Big(1+\eta(\eps_1D_1D_3,\eps_2D_2D_3)\Big) \\
&\phantom{blahblah}\times\sum_{D_i=D_{i1}D_{i2}}\Big(\frac{\eps_1}{D_{21}D_{31}}\Big)\Big(\frac{\eps_2}{D_{11}D_{31}}\Big)\Big(\frac{-1}{D_{11}D_{21}}\Big)\Phi(D_{11},D_{21},D_{31})2^{-\omega(\vec{D})}g(\vec{D}).
\end{align*}
Next we evaluate the products $\ind_{M_i}\ind_{M_j}$. As we did above, we take a product over local conditions, which gives us
\begin{align*}
&\ind_{M_1}\ind_{M_2}(d_1,d_2,d_3)\;=\;\ind_{\sign(d_1,d_2)\ne(\pm1,\mp1)}\cdot\frac{1}{2}\Big(1+\eta(d_1+2,d_2)\Big)\frac{1}{2}\Big(1+\eta(d_1,d_2+2)\Big) \\
&\times \prod_{\substack{p\mid d_1 \\ p\nmid d_2}}\frac{1}{2}\Big(1+\Big(\frac{d_2}{p}\Big)\Big)\frac{1}{2}\Big(1+\Big(\frac{-d_2}{p}\Big)\Big)\prod_{\substack{p\nmid d_1 \\ p\mid d_2}}\frac{1}{2}\Big(1+\Big(\frac{d_1}{p}\Big)\Big)\frac{1}{2}\Big(1+\Big(\frac{-d_1}{p}\Big)\Big)\prod_{\substack{p\mid d_1 \\ p\mid d_2}}\frac{1}{2}\Big(1+\Big(\frac{d_{1,p}d_{2,p}}{p}\Big)\Big)^2, \\
&\phantom{\ind_{M_1}\ind_{M_2}(d_1,d_2,d_3)\;}=\;\ind_{\sign(d_1,d_2)\ne(\pm1,\mp1)}\cdot\frac{1}{2}\Big(1+\eta(d_1+2,d_2)\Big)\frac{1}{2}\Big(1+\eta(d_1,d_2+2)\Big) \\
&\times \prod_{\substack{p\mid d_1 \\ p\nmid d_2}}\frac{1}{2}\Big(1+\Big(\frac{d_2}{p}\Big)\Big)\frac{1}{2}\Big(1+\Big(\frac{-1}{p}\Big)\Big)\prod_{\substack{p\nmid d_1 \\ p\mid d_2}}\frac{1}{2}\Big(1+\Big(\frac{d_1}{p}\Big)\Big)\frac{1}{2}\Big(1+\Big(\frac{-1}{p}\Big)\Big)\prod_{\substack{p\mid d_1 \\ p\mid d_2}}\frac{1}{2}\Big(1+\Big(\frac{d_{1,p}d_{2,p}}{p}\Big)\Big).
\end{align*}
Just as above we put $d_1=\eps_1D_1D_3$ and $d_2=\eps_2D_2D_3$ with $D_3=\gcd(d_1,d_2)>0$ and $\sign(d_i)=\eps_i$. Following the same process as above, we end up with
\begin{align*}
\ind_{M_1}\ind_{M_2}(\eps_1D_1D_3,&\eps_2D_2D_3,\eps_1\eps_2D_1D_2)\;=\;\ind_{(\eps_1,\eps_2)\ne(\pm1,\mp1)}\cdot\frac{1}{2}\Big(1+\eta(\eps_1D_1D_3+2,\eps_2D_2D_3)\Big) \\
&\phantom{blah}\times\frac{1}{2}\Big(1+\eta(\eps_1D_1D_3,\eps_2D_2D_3+2)\Big)\sum_{D_i=D_{i1}D_{i2}}\psi(D_1D_2)\Big(\frac{\eps_1\eps_2}{D_{31}}\Big)2^{-\omega(\vec{D})}g(\vec{D}).
\end{align*}
Note that in deriving the above, we have used the fact that
\[
\Big(\frac{\eps_2}{D_{11}D_{31}}\Big)\Big(\frac{\eps_1}{D_{21}D_{31}}\Big)\Big(\frac{-1}{D_{11}D_{21}D_{31}}\Big)\Phi(D_{11},D_{21},D_{31})\;=\;\Big(\frac{\eps_1\eps_2}{D_{31}}\Big)
\]
because $D_{11}\equiv D_{21}\equiv \mymod{1}{4}$ in this case. The only nonzero terms in the above have $\eps_1=\eps_2$, and hence $(\eps_1\eps_2/D_{31})=1$ for these terms. Furthermore, using the conditions $\eps_1=\eps_2$ together with the conditions $D_1\equiv D_2\equiv \mymod{1}{4}$, we have
\[
\frac{1}{2}\Big(1+\eta(\eps_1D_1D_3+2,\eps_2D_2D_3)\Big)\frac{1}{2}\Big(1+\eta(\eps_1D_1D_3,\eps_2D_2D_3+2)\Big)\;=\;1.
\]
Thus we have
\[
\ind_{M_1}\ind_{M_2}(\eps_1D_1D_3,\eps_2D_2D_3,\eps_1\eps_2D_1D_2)\;=\;\ind_{(\eps_1,\eps_2)\ne(\pm1,\mp1)}\sum_{D_i=D_{i1}D_{i2}}\psi(D_1D_2)2^{-\omega(\vec{D})}g(\vec{D}).
\]
We follow the same steps for the terms $\ind_{M_1}\ind_{M_3}$ and $\ind_{M_2}\ind_{M_3}$, and we derive
\begin{align*}
\ind_{M_1}\ind_{M_3}(\eps_1D_1D_3,\eps_2D_2D_3,\eps_1\eps_2D_1D_2)\;&=\;\ind_{(\eps_1,\eps_2)\ne(\pm1,-1)}\sum_{D_i=D_{i1}D_{i2}}\psi(D_2D_3)2^{-\omega(\vec{D})}g(\vec{D}), \\
\ind_{M_2}\ind_{M_3}(\eps_1D_1D_3,\eps_2D_2D_3,\eps_1\eps_2D_1D_2)\;&=\;\ind_{(\eps_1,\eps_2)\ne(-1,\pm1)}\sum_{D_i=D_{i1}D_{i2}}\psi(D_1D_3)2^{-\omega(\vec{D})}g(\vec{D}),
\end{align*}
and for $\ind_{M_1}\ind_{M_2}\ind_{M_3}$ we get
\[
\ind_{M_1}\ind_{M_2}\ind_{M_3}(\eps_1D_1D_3,\eps_2D_2D_3,\eps_1\eps_2D_1D_2)\;=\;\ind_{(\eps_1,\eps_2)=(+1,+1)}\sum_{D_i=D_{i1}D_{i2}}\psi(D_1D_2D_3)2^{-\omega(\vec{D})}g(\vec{D}).
\]
Putting everything together using the identity \eqref{d4-to-mi} gives the result.

Now we have the case $2\mid\gcd(d_1,d_2)$. We proceed as before, taking a product over the local conditions from Lemma \ref{lem:d4}, which gives us
\begin{align*}
\ind_{M_1}(d_1,d_2,d_3)\;&=\;\ind_{\sign(d_1,d_2)\ne(+1,-1)}\cdot\frac{1}{2}\Big(1+\eta(d_{1,1}+2,d_{2,2})\Big(\frac{2}{d_{1,2}d_{2,2}}\Big)\Big) \\
&\phantom{blahblah}\times \prod_{\substack{p\mid d_1 \\ p\nmid d_2 \\ p>2}}\frac{1}{2}\Big(1+\Big(\frac{d_2}{p}\Big)\Big)\prod_{\substack{p\nmid d_1 \\ p\mid d_2 \\ p>2}}\frac{1}{2}\Big(1+\Big(\frac{-d_1}{p}\Big)\Big)\prod_{\substack{p\mid d_1 \\ p\mid d_2 \\ p>2}}\frac{1}{2}\Big(1+\Big(\frac{d_{1,p}d_{2,p}}{p}\Big)\Big),
\\
\ind_{M_2}(d_1,d_2,d_3)\;&=\;\ind_{\sign(d_1,d_2)\ne(-1,+1)}\cdot\frac{1}{2}\Big(1+\eta(d_{1,2},d_{2,2}+2)\Big(\frac{2}{d_{1,2}d_{2,2}}\Big)\Big) \\
&\phantom{blahblah}\times \prod_{\substack{p\mid d_1 \\ p\nmid d_2 \\ p>2}}\frac{1}{2}\Big(1+\Big(\frac{-d_2}{p}\Big)\Big)\prod_{\substack{p\nmid d_1 \\ p\mid d_2 \\ p>2}}\frac{1}{2}\Big(1+\Big(\frac{d_1}{p}\Big)\Big)\prod_{\substack{p\mid d_1 \\ p\mid d_2 \\ p>2}}\frac{1}{2}\Big(1+\Big(\frac{d_{1,p}d_{2,p}}{p}\Big)\Big),
\end{align*}
\begin{align*}
\ind_{M_3}(d_1,d_2,d_3)\;&=\;\ind_{\sign(d_1,d_2)\ne(-1,-1)}\cdot\frac{1}{2}\Big(1+\eta(d_{1,2},d_{2,2})\Big(\frac{2}{d_{1,2}d_{2,2}}\Big)\Big) \\
&\phantom{blahblah}\times \prod_{\substack{p\mid d_1 \\ p\nmid d_2 \\ p>2}}\frac{1}{2}\Big(1+\Big(\frac{d_2}{p}\Big)\Big)\prod_{\substack{p\nmid d_1 \\ p\mid d_2 \\ p>2}}\frac{1}{2}\Big(1+\Big(\frac{d_1}{p}\Big)\Big)\prod_{\substack{p\mid d_1 \\ p\mid d_2 \\ p>2}}\frac{1}{2}\Big(1+\Big(\frac{-d_{1,p}d_{2,p}}{p}\Big)\Big).
\end{align*}
In each case we put $d_1=\eps_12D_1D_3$ and $d_2=\eps_22D_2D_3$, where $D_3=\gcd(d_1,d_2)/2>0$ and $\eps_i=\sign(d_i)$, so that the $D_i$ are odd, positive, squarefree, and pairwise coprime. Following the same steps as above, we get
\begin{align*}
\ind_{M_1}(\eps_12D_1D_3,&\eps_22D_2D_3,\eps_1\eps_2D_1D_2)\;=\;\ind_{(\eps_1,\eps_2)\ne(+1,-1)}\cdot\frac{1}{2}\Big(1+\eta(\eps_1D_1D_3+2,\eps_2D_2D_3)\Big(\frac{2}{D_1D_2}\Big)\Big) \\
&\phantom{bl}\times\sum_{D_i=D_{i1}D_{i2}}\Big(\frac{\eps_1}{D_{21}D_{31}}\Big)\Big(\frac{\eps_2}{D_{11}D_{31}}\Big)\Big(\frac{-1}{D_{11}D_{31}}\Big)\Big(\frac{2}{D_{11}D_{21}}\Big)\Phi(D_{11},D_{21},D_{31})2^{-\omega(\vec{D})}g(\vec{D}), \\
\ind_{M_2}(\eps_12D_1D_3,&\eps_22D_2D_3,\eps_1\eps_2D_1D_2)\;=\;\ind_{(\eps_1,\eps_2)\ne(-1,+1)}\cdot\frac{1}{2}\Big(1+\eta(\eps_1D_1D_3,\eps_2D_2D_3+2)\Big)\Big(\frac{2}{D_1D_2}\Big)\Big) \\
&\phantom{bl}\times\sum_{D_i=D_{i1}D_{i2}}\Big(\frac{\eps_1}{D_{21}D_{31}}\Big)\Big(\frac{\eps_2}{D_{11}D_{31}}\Big)\Big(\frac{-1}{D_{21}D_{31}}\Big)\Big(\frac{2}{D_{11}D_{21}}\Big)\Phi(D_{11},D_{21},D_{31})2^{-\omega(\vec{D})}g(\vec{D}), \\
\ind_{M_3}(\eps_12D_1D_3,&\eps_22D_2D_3,\eps_1\eps_2D_1D_2)\;=\;\ind_{(\eps_1,\eps_2)\ne(-1,-1)}\cdot\frac{1}{2}\Big(1+\eta(\eps_1D_1D_3,\eps_2D_2D_3)\Big)\Big(\frac{2}{D_1D_2}\Big)\Big) \\
&\phantom{bl}\times\sum_{D_i=D_{i1}D_{i2}}\Big(\frac{\eps_1}{D_{21}D_{31}}\Big)\Big(\frac{\eps_2}{D_{11}D_{31}}\Big)\Big(\frac{-1}{D_{11}D_{21}}\Big)\Big(\frac{2}{D_{11}D_{21}}\Big)\Phi(D_{11},D_{21},D_{31})2^{-\omega(\vec{D})}g(\vec{D}).
\end{align*}
We follow the same steps as we did above to evaluate the products of the $\ind_{M_i}$, and we get
\begin{align*}
\ind_{M_1}\ind_{M_2}(\eps_12D_1D_3,\eps_22D_2D_3,\eps_1\eps_2D_1D_2)\;&=\;\ind_{(\eps_1,\eps_2)\ne(\pm1,-1)}\frac{1}{2}\Big(1+\Big(\frac{2}{D_1D_2}\Big)\Big) \\
&\phantom{blahblah}\times\sum_{D_i=D_{i1}D_{i2}}\psi(D_2D_3)2^{-\omega(\vec{D})}g(\vec{D}), \\
\ind_{M_1}\ind_{M_3}(\eps_12D_1D_3,\eps_22D_2D_3,\eps_1\eps_2D_1D_2)\;&=\;\ind_{(\eps_1,\eps_2)\ne(\pm1,-1)}\frac{1}{2}\Big(1+\Big(\frac{2}{D_1D_2}\Big)\Big) \\
&\phantom{blahblah}\times\sum_{D_i=D_{i1}D_{i2}}\psi(D_2D_3)2^{-\omega(\vec{D})}g(\vec{D}), \\
\ind_{M_2}\ind_{M_3}(\eps_12D_1D_3,\eps_22D_2D_3,\eps_1\eps_2D_1D_2)\;&=\;\ind_{(\eps_1,\eps_2)\ne(-1,\pm1)}\frac{1}{2}\Big(1+\Big(\frac{2}{D_1D_2}\Big)\Big) \\
&\phantom{blahblah}\times\sum_{D_i=D_{i1}D_{i2}}\psi(D_1D_3)2^{-\omega(\vec{D})}g(\vec{D}), \\
\ind_{M_1}\ind_{M_2}\ind_{M_3}(\eps_12D_1D_3,\eps_22D_2D_3,\eps_1\eps_2D_1D_2)\;&=\;\ind_{(\eps_1,\eps_2)=(+1,+1)}\frac{1}{2}\Big(1+\Big(\frac{2}{D_1D_2}\Big)\Big) \\
&\phantom{blahblah}\times\sum_{D_i=D_{i1}D_{i2}}\psi(D_1D_2D_3)2^{-\omega(\vec{D})}g(\vec{D}).
\end{align*}
Once again, putting everything together gives the result.
\end{proof}

\section{Averaging over discriminants}
  \label{sec:average}

In this section we use our results from the previous section to derive
expressions for the quantities
$
B_h^\sigma(X;G)
$
(defined in \eqref{bhxG}) that will be amenable to our subsequent analysis. As
we have seen above, there are many cases and sub-cases to consider, and it
would be tedious and convoluted to treat their analyses individually.
Therefore our goal is to uniformize our expressions as much as possible and
give essentially one common form of the expressions for the different functions
$
B_h^\sigma(X;G)
$.

We will use our notation from the previous section,
\[
\vec{D}\;=\;(D_{11},D_{12},D_{21},D_{22},D_{31},D_{32}),
\]
and recall that the variables
$
D_{ij}
$
always denote positive, odd, squarefree integers. Moreover, the variables
$
D_{ij}
$
are always pairwise coprime, regardless of how they are defined in terms of the
$
d_i
$
variables. We use the notations $\omega(\vec{D})$ and $g(\vec{D})$ as defined
in \eqref{omega-of-d} and \eqref{g-of-d}, and now we also put
\[
\Delta(\vec{D})\;:=\;(D_{11}D_{12}D_{21}D_{22}D_{31}D_{32})^2.
\]
We will want to consider the tuple $\vec{D}$ modulo 8, so we put $H=(\Z/8\Z)^\times$; for $\vec{w}\in H^6$, we will write $\vec{D}\equiv\vec{w}\;(8)$ to mean $D_{ij}\equiv\mymod{w_{ij}}{8}$ for all $1\le i\le3$, $1\le j\le 2$, where
\[
\vec{w}\;=\;(w_{11},w_{12},w_{21},w_{22},w_{31},w_{32}).
\]
Lastly, recall the quantities $m_h$ and $\tau_h$ defined in Section \ref{sec:parametrize}, which are defined for 
$
1\le h\le 4
$:
 $m_1=6$, $m_2=m_3=m_4=2$, and $\tau_1=1$, $\tau_2=\tau_3=16$, and $\tau_4=64$. Here we now also put
\[
T_h\;=\;
\begin{cases}
  \tau_h & \text{if }h=1\text{ or }3,
  \\
  4\tau_h & \text{if }h=2\text{ or }4;
\end{cases}
\]
explicitly we have $T_1=1$, $T_2=64$, $T_3=16$, and $T_4=256$.

In principle we are just summing over admissible triples
$
(d_1,d_2,d_3)
$,
and we use the notation $\flatsum$ to denote this condition on the triples
$
(d_1,d_2,d_3)
$.
However, we will then have to factor the variables
$
d_i
$
for
$
1\le i\le3
$,
first into variables
$
D_i, 1\le i\le 3
$,
and then further into variables
$
D_{ij}
$
where
$
1\le i\le 3,  1\le j\le 2
$.
The actual factorization depends on the particular case we are treating
(i.e.~depending on $G=Q_8$ or $D_4$, $1\le h\le 4$, and the signs $\sigma$).
However, as we factor variables and rearrange the summations, we will continue to use the same notation $\flatsum$ to refer to the conditions that, in any particular case, ensure that the variables of summation run through values such that $(d_1,d_2,d_3)$ is admissible. We hope that this abuse of notation is not confusing but in fact helpful to the reader, conveying the manipulations in a readable manner that is not overcrowded with notation.

In many cases, the implicit conditions are the usual ones (that we sum over positive, odd, squarefree integers), along with pairwise coprimality of the variables (after factoring the $d_i$), and that certain products of variables not be equal to $1$. As we go through the proof, we remind the reader of these conditions so as to help them not lose their bearing. 

At the end of our calculations (and thus in the statement of the below
propositions) we use the notation
$
\natsum_{\vec{D}}
$
to denote a sum over tuples of integers that are positive, odd, squarefree, and pairwise coprime---that is, we drop any additional conditions that guarantee that the underlying triples $(d_1,d_2,d_3)$ are admissible. This is for uniformity in our expressions and for simplicity in the analytic treatments to follow.

\begin{prop}[$Q_8$ case]
    \label{prop:fhq8}
If
$
\sigma\ne(+1,+1)
$,
we have
$
{B}_h^\sigma(X;Q_8)=0
$
for every
$
1\le h\le 4
$.
Otherwise, for
$
\sigma=(+1,+1)
$ we have
\begin{equation}
\label{q8case}
  {B}_h^\sigma(X;Q_8)
  \;=\;
  \frac{1}
       {m_h}
  \sum_{\vec{w}\in H^6}
       f_h(\vec{w};Q_8)
       \natsum_{\substack{
                            \Delta(\vec{D})\le X/T_h
                            \\
                            \vec{D}\;\equiv\;\vec{w}\;(8)
                         }
               }
       2^{-\omega(\vec{D})}
       g(\vec{D})
  \;+\;
    O(\sqrt{X}),
\end{equation}
where the functions
$
f_h(-;Q_8):H^6\to\{+1,0,-1\}
$
are given explicitly by
\begin{align*}
  f_1(\vec{w};Q_8)
  \;&:=\;
  \frac{1}{16}
  \Big(
      1+\Big(\frac{-1}{w_{11}w_{12}w_{31}w_{32}}\Big)
  \Big)
  \Big(
      1+\Big(\frac{-1}{w_{21}w_{22}w_{31}w_{32}}\Big)
  \Big)
      \Big(1+\Big(\frac{-1}{w_{11}w_{12}w_{21}w_{22}}\Big)
  \Big)
  \\
  &
  \phantom{blahblah}\times\Big(1+\eta(w_{11}w_{12}w_{31}w_{32},w_{21}w_{22}w_{31}w_{32})\Big(\frac{-1}{w_{11}w_{12}w_{21}w_{22}}\Big)\Big)\Phi( w_{11}, w_{21}, w_{31} ),
\end{align*}
\begin{align*}
  f_2(\vec{w};Q_8)
  \;&:=\;
  \frac{1}
       {4}
  \Big(
    1
    +
    \Big(
       \frac{-1}
            {w_{11}w_{12}w_{21}w_{22}}
    \Big)
  \Big)
  \Big(
     \frac{2}
          {w_{11}w_{21}}
  \Big)
  \\
  &
  \phantom{blahblah}
  \times
  \Big(
    1
    +
    \eta(w_{11}w_{12}w_{31}w_{32},w_{21}w_{22}w_{31}w_{32})
    \Big(
       \frac{-2}
            {w_{11}w_{12}w_{21}w_{22}}
    \Big)
  \Big)
  \Phi( w_{11}, w_{21}, w_{31} ),
  \\
  f_3(\vec{w};Q_8)\;&:=\;\frac{1}{16}\Big(1-\Big(\frac{-1}{w_{11}w_{12}w_{31}w_{32}}\Big)\Big)\Big(1-\Big(\frac{-1}{w_{21}w_{22}w_{31}w_{32}}\Big)\Big)\Big(1+\Big(\frac{-1}{w_{11}w_{12}w_{21}w_{22}}\Big)\Big)
  \\
  &\phantom{blahblah}\times\Big(1+\eta(w_{11}w_{12}w_{31}w_{32},w_{21}w_{22}w_{31}w_{32})\Big(\frac{-1}{w_{11}w_{12}w_{21}w_{22}}\Big)\Big)\Phi( w_{11}, w_{21}, w_{31} ),
  \\
  f_4(\vec{w};Q_8)\;&:=\;\frac{1}{4}\Big(1-\Big(\frac{-1}{w_{11}w_{12}w_{21}w_{22}}\Big)\Big)\Big(\frac{2}{w_{11}w_{21}}\Big)
  \\
  &
  \phantom{blahblah}\times\Big(1+\eta(w_{11}w_{12}w_{31}w_{32},w_{21}w_{22}w_{31}w_{32})\Big(\frac{-2}{w_{11}w_{12}w_{21}w_{22}}\Big)\Big)\Phi( w_{11}, w_{21}, w_{31} ).
\end{align*}
\end{prop}

\begin{proof}
For
$
{B}_1(X;Q_8)
$,
the admissible triples satisfy $d_i\equiv\mymod{1}{4}$, and so we have
\begin{align*}
  {B}_1(X;Q_8)
  \;&=\;
  \frac{1}
       {m_1}
  \sum_{(d_1,d_2,d_3)\in \mathcal{S}_1(X)}  \ind_{Q_8}(d_1,d_2,d_3)
  \\
  &=\;
  \frac{1}
       {m_1}
  \flatsum_{\tau_1d_1d_2d_3\le X}
  \frac{1}
       {2}
  \Big(
       1+\Big(\frac{-1}{d_1}\Big)
  \Big)
  \frac{1}
       {2}
  \Big(
       1+\Big(\frac{-1}{d_2}\Big)
  \Big)
  \frac{1}
       {2}
  \Big(
       1+\Big(\frac{-1}{d_3}\Big)
  \Big)
  \ind_{Q_8}(d_1,d_2,d_3),
\end{align*}
where we write
  $
  \sum^\flat
  $
  to indicate that the sum is taken over admissible triples $(d_1,d_2,d_3)$.
Using
the result of Lemma \ref{lem:q8-embedding-function} (in the case when $2\nmid d_1d_2$), we have
\begin{align*}
  {B}_1(X;Q_8)
  \;=\;
  \frac{1}{m_1}
  \flatsum_{\tau_1(D_1D_2D_3)^2\le X}
  &
  \frac{1}{8}\Big(1+\Big(\frac{-1}{D_1D_3}\Big)\Big)\Big(1+\Big(\frac{-1}{D_2D_3}\Big)\Big)\Big(1+\Big(\frac{-1}{D_1D_2}\Big)\Big)
  \\
  &
  \times \sum_{D_i=D_{i1}D_{i2}} F_{\odd}(\vec{D};Q_8)2^{-\omega(\vec{D})}g(\vec{D}), 
\end{align*}
where again
$
\sum^\flat
$
denotes that the sum runs over positive, odd, squarefree integers $D_1,D_2,D_3$ such that
$
(D_1D_3,D_2D_3,D_1D_2)
$
is an admissible triple---explicitly, this means that the $D_i$ are pairwise coprime, and no two of them are simultaneously equal to $1$. We now interchange the order of summation. The above then becomes
\[
   {B}_1(X;Q_8)
   \;=\;
   \frac{1}
        {m_1}
   \flatsum_{\tau_1\Delta(\vec{D})\le X}
        2^{-\omega(\vec{D})}
        \widetilde{f}(\vec{D})
        g(\vec{D}),
\]
where the sum is taken over tuples
$
\vec{D}=(D_{11},D_{12},D_{21},D_{22},D_{31},D_{32})
$
such that
\[
  (D_{11}D_{12}D_{31}D_{32},D_{21}D_{22}D_{31}D_{32},D_{11}D_{12}D_{21}D_{22})
\] is an admissible triple, and where we have put
\begin{align*}
\widetilde{f}(\vec{D})\;:&=\;\frac{1}{8}\Big(1+\Big(\frac{-1}{D_1D_3}\Big)\Big)\Big(1+\Big(\frac{-1}{D_2D_3}\Big)\Big)\Big(1+\Big(\frac{-1}{D_1D_2}\Big)\Big)F_{\odd}(\vec{D};Q_8) \\
&=\;\frac{1}{16}\Big(1+\Big(\frac{-1}{D_1D_3}\Big)\Big)\Big(1+\Big(\frac{-1}{D_2D_3}\Big)\Big)\Big(1+\Big(\frac{-1}{D_1D_2}\Big)\Big) \\
&\phantom{blahblah}\times\Big(1+\eta(D_{1}D_{3},D_2D_3)\Big(\frac{-1}{D_1D_2}\Big)\Big)\Phi( D_{11}, D_{21}, D_{31} ).
\end{align*}
Importantly, the function $\widetilde{f}$ above only depends on the tuple $\vec{D}$ modulo 8. Moreover, the variables $D_{ij}$ are odd. Therefore there exists a function $f_1(-;Q_8):H^6\to\{+1,0,-1\}$ such that $f_1(\vec{w};Q_8)=\widetilde{f}(\vec{D})$ if $\vec{D}\equiv\mymod{\vec{w}}{8}$; indeed, this is the function $f_1(\vec{w};Q_8)$ given in the statement of the proposition. This allows us to write
\[
{B}_1(X;Q_8)\;=\;\frac{1}{m_1}\sum_{\vec{w}\in H^6}f_1(\vec{w};Q_8)\flatsum_{\substack{\tau_1\Delta(\vec{D})\le X \\ \vec{D}\;\equiv\;\vec{w}\;(8) }}2^{-\omega(\vec{D})}g(\vec{D}).
\]
Finally, we adjust the conditions of summation for convenience in the analytic arguments to follow.
We use the notation
$
\sum^\natural
$
to indicate a sum taken over tuples $\vec{D}$ whose entries are odd, positive, squarefree integers that are pairwise coprime---that is, we drop the conditions in the $\sum^\flat$ summations above that certain products cannot be equal to 1. Terms that appear in such a $\sum^\natural$ summation but not in a $\sum^\flat$ one must have at least four
$
D_{ij}=1
$,
and by Lemma~\ref{sqfree-divisor-bound} in Section \ref{sec:lemma}, the total contribution of such terms is
$
O(\sqrt{X})
$. Therefore we have
\[
{B}_1(X;Q_8)\;=\;\frac{1}{m_1}\sum_{\vec{w}\in H^6}f_1(\vec{w};Q_8)\natsum_{\substack{\tau_1\Delta(\vec{D})\le X \\ \vec{D}\;\equiv\;\vec{w}\;(8) }}2^{-\omega(\vec{D})}g(\vec{D})\;+\;O(\sqrt{X}).
\]
Next, we treat ${B}_2$ similarly: we have
\[
{B}_2(X;Q_8)\;=\;\frac{1}{m_2}\flatsum_{\substack{\tau_2d_1d_2d_3\le X \\ d_1\equiv d_2\equiv2\;(4)}}\frac{1}{2}\Big(1+\Big(\frac{-1}{d_3}\Big)\Big)\ind_{Q_8}(d_1,d_2,d_3),
\]
where again $\sum^\flat$ denotes summation over admissible triples. As above, we
apply Lemma \ref{lem:q8-embedding-function} in the case when
$
2\mid\gcd(d_1,d_2)
$
and put $d_1=2D_1D_3$, $d_2=2D_2D_3$; that gives
\begin{align*}
{B}_2(X;Q_8)\;=\;\frac{1}{m_2}&\flatsum_{4\tau_2(D_1D_2D_3)^2\le X}\frac{1}{4}\Big(1+\Big(\frac{-1}{D_1D_2}\Big)\Big) 2^{-\omega(D_1 D_2 D_3)}
     \Big(1+\eta(D_{1}D_{3},D_2D_3)\Big(\frac{-2}{D_1D_2}\Big)\Big)
\\
   &\phantom{blahbl}\times \sum_{D_{i1}\mid D_i}
   \Big(\frac{2}{D_{11}D_{21}}\Big)
      \Phi( D_{11}, D_{21}, D_{31} )
      \Bigl(
        \frac{ D_{22} D_{32} }{ D_{11} }
      \Bigl)
      \Bigl(
        \frac{ D_{12} D_{32} }{ D_{21} }
      \Bigl)
      \Bigl(
        \frac{ D_{12} D_{22} }{ D_{31} }
      \Bigl).
\end{align*}
Following the same lines as in the above, we deduce that
\[
{B}_2(X;Q_8)\;=\;\frac{1}{m_2}\sum_{\vec{w}\in H^6}f_2(\vec{w};Q_8)\natsum_{\substack{4\tau_2\Delta(\vec{D})\le X \\ \vec{D}\;\equiv\;\vec{w}\;(8) }}2^{-\omega(\vec{D})}g(\vec{D})\;+\;O(\sqrt{X}),
\]
where again $f_2(\vec{w};Q_8)$ is defined in the statement of the proposition. We handle ${B}_3$ and ${B}_4$ in the same way, which completes the proof.
\end{proof}

\begin{remm}
  Of course, we know already from Corollary \ref{cor:b3q8} that
  $
  {B}_3(X;Q_8)=0
  $,
  so we have no need for the formula for it in Proposition \ref{prop:fhq8}.
  Nevertheless we keep the formula there for completeness' sake.
  The fact that ${B}_3(X;Q_8)=0$ is later reflected ``analytically'' in the proof of Theorem \ref{thm:main-thm}, when (in the notation there) we find that $S_3^\sigma(Q_8)=0$.
\end{remm}

Next we give the corresponding proposition for the $D_4$ case. While the expressions that follow (especially the explicit formulas for the functions $f_h^\sigma$ and $r_h^\sigma$) look more complicated, these intricacies are  superficial. The essential feature is that equations \eqref{q8case} and \eqref{d4case} are of a nearly identical shape analytically speaking, and this allows for a uniform analytic treatment in the sections to follow. The only significant difference is the presence of the summation over the function $r_h^\sigma$ in \eqref{d4case}, which has no analogue in \eqref{q8case}. However, we will show in Section \ref{sec:asymptotic_evaluation_of_b_h_sigma_x_g_} that these terms contribute negligibly to the quantity $B_h^\sigma(X;D_4)$.

\begin{prop}[$D_4$ case]
  \label{prop:fhd4}
  Write
  $
  \sigma=(\eps_1,\eps_2)
  $,
  where each
  $
  \eps_i\in\{\pm1\}
  $.
  We have
\begin{align}
\label{d4case}
  B_h^\sigma(X;D_4)
  \;&=\;
  \frac{1}{m_h}\sum_{\vec{w}\in H^6}f_h^\sigma(\vec{w};D_4)\natsum_{\substack{\Delta(\vec{D})\le X/T_h \\ \vec{D}\;\equiv\;\vec{w}\;(8) }}2^{-\omega(\vec{D})}g(\vec{D})\\
&\phantom{blahblah}+\natsum_{\Delta(\vec{D})\le X/T_h}r_h^\sigma(\vec{D};D_4)2^{-\omega(\vec{D})}g(\vec{D})\;+\;O(\sqrt{X}), \nonumber
\end{align}
where the functions $f_h^\sigma$ are bounded, integer-valued functions on $H^6$ (we have $|f_h^\sigma(\vec{w};D_4)|\le3$ for all $\vec{w}\in H^6$), and they are given explicitly by
\begin{align*}
  &
  f_1^\sigma(\vec{w};D_4)\;:=\;\frac{1}{8}\Big(1+\eps_1\Big(\frac{-1}{w_{11}w_{12}w_{31}w_{32}}\Big)\Big)\Big(1+\eps_2\Big(\frac{-1}{w_{21}w_{22}w_{31}w_{32}}\Big)\Big)
  \\
  &
  \phantom{f_1^\sigma(\vec{w};D_4)blah}\times\;\Big(1+\eps_1\eps_2\Big(\frac{-1}{w_{11}w_{12}w_{21}w_{22}}\Big)\Big)\Big(\frac{\eps_1}{w_{21}w_{31}}\Big)\Big(\frac{\eps_2}{w_{11}w_{31}}\Big)\Phi(w_{11},w_{21},w_{31})
  \\
  &
  \phantom{f_1^\sigma(\vec{w};D_4)blah}\times\;\bigg\{\ind_{\sigma\ne(+1,-1)}\cdot\frac{1}{2}\Big(1+\eta(\eps_1w_{11}w_{12}w_{31}w_{32}+2,\eps_2w_{21}w_{22}w_{31}w_{32})\Big)\Big(\frac{-1}{w_{11}w_{31}}\Big)
  \\
  &
  \phantom{f_1^\sigma(\vec{w};D_4)blahblah}+\ind_{\sigma\ne(-1,+1)}\cdot\frac{1}{2}\Big(1+\eta(\eps_1w_{11}w_{12}w_{31}w_{32},\eps_2w_{21}w_{22}w_{31}w_{32}+2)\Big)\Big(\frac{-1}{w_{21}w_{31}}\Big)
  \\
  &
  \phantom{f_1^\sigma(\vec{w};D_4)blahblah}+\ind_{\sigma\ne(-1,-1)}\cdot\frac{1}{2}\Big(1+\eta(\eps_1w_{11}w_{12}w_{31}w_{32},\eps_2w_{21}w_{22}w_{31}w_{32})\Big)\Big(\frac{-1}{w_{11}w_{21}}\Big)\bigg\},
\end{align*}
\begin{align*}
  &
  f_2^\sigma(\vec{w};D_4)\;:=\;\frac{1}{2}\Big(1+\eps_1\eps_2\Big(\frac{-1}{w_{11}w_{12}w_{21}w_{22}}\Big)\Big)\Big(\frac{\eps_1}{w_{21}w_{31}}\Big)\Big(\frac{\eps_2}{w_{11}w_{31}}\Big)\Big(\frac{2}{w_{11}w_{21}}\Big)\Phi(w_{11},w_{21},w_{31})
  \\
  &\phantom{blah}\times\;\bigg\{\ind_{\sigma\ne(+1,-1)}\cdot\frac{1}{2}\Big(1+\eta(\eps_1w_{11}w_{12}w_{31}w_{32}+2,\eps_2w_{21}w_{22}w_{31}w_{32})\Big(\frac{2}{w_{11}w_{12}w_{21}w_{22}}\Big)\Big)\Big(\frac{-1}{w_{11}w_{31}}\Big)
  \\
  &
  \phantom{blahblah}+\ind_{\sigma\ne(-1,+1)}\cdot\frac{1}{2}\Big(1+\eta(\eps_1w_{11}w_{12}w_{31}w_{32},\eps_2w_{21}w_{22}w_{31}w_{32}+2)(\frac{2}{w_{11}w_{12}w_{21}w_{22}}\Big)\Big)\Big(\frac{-1}{w_{21}w_{31}}\Big)
  \\
  &
  \phantom{blahblah}+\ind_{\sigma\ne(-1,-1)}\cdot\frac{1}{2}\Big(1+\eta(\eps_1w_{11}w_{12}w_{31}w_{32},\eps_2w_{21}w_{22}w_{31}w_{32})(\frac{2}{w_{11}w_{12}w_{21}w_{22}}\Big)\Big)\Big(\frac{-1}{w_{11}w_{21}}\Big)\bigg\},
\end{align*}
\begin{align*}
  &
  f_3^\sigma(\vec{w};D_4)\;:=\;\frac{1}{8}\Big(1-\eps_1\Big(\frac{-1}{w_{11}w_{12}w_{31}w_{32}}\Big)\Big)\Big(1-\eps_2\Big(\frac{-1}{w_{21}w_{22}w_{31}w_{32}}\Big)\Big)
  \\
  &\phantom{f_3^\sigma(\vec{w};D_4)blah}\times\;\Big(1+\eps_1\eps_2\Big(\frac{-1}{w_{11}w_{12}w_{21}w_{22}}\Big)\Big)\Big(\frac{\eps_1}{w_{21}w_{31}}\Big)\Big(\frac{\eps_2}{w_{11}w_{31}}\Big)\Phi(w_{11},w_{21},w_{31})
  \\
  &\phantom{f_3^\sigma(\vec{w};D_4)blah}\times\;\bigg\{\ind_{\sigma\ne(+1,-1)}\cdot\frac{1}{2}\Big(1+\eta(\eps_1w_{11}w_{12}w_{31}w_{32}+2,\eps_2w_{21}w_{22}w_{31}w_{32})\Big)\Big(\frac{-1}{w_{11}w_{31}}\Big)
  \\
  &
  \phantom{f_3^\sigma(\vec{w};D_4)blahblah}+\ind_{\sigma\ne(-1,+1)}\cdot\frac{1}{2}\Big(1+\eta(\eps_1w_{11}w_{12}w_{31}w_{32},\eps_2w_{21}w_{22}w_{31}w_{32}+2)\Big)\Big(\frac{-1}{w_{21}w_{31}}\Big)
  \\
  &\phantom{f_3^\sigma(\vec{w};D_4)blahblah}+\ind_{\sigma\ne(-1,-1)}\cdot\frac{1}{2}\Big(1+\eta(\eps_1w_{11}w_{12}w_{31}w_{32},\eps_2w_{21}w_{22}w_{31}w_{32})\Big)\Big(\frac{-1}{w_{11}w_{21}}\Big)\bigg\},
\end{align*}
\begin{align*}
  &
  f_4^\sigma(\vec{w};D_4)\;:=\;\frac{1}{2}\Big(1-\eps_1\eps_2\Big(\frac{-1}{w_{11}w_{12}w_{21}w_{22}}\Big)\Big)\Big(\frac{\eps_1}{w_{21}w_{31}}\Big)\Big(\frac{\eps_2}{w_{11}w_{31}}\Big)\Big(\frac{2}{w_{11}w_{21}}\Big)\Phi(w_{11},w_{21},w_{31}) \\
&\phantom{blah}\times\;\bigg\{\ind_{\sigma\ne(+1,-1)}\cdot\frac{1}{2}\Big(1+\eta(\eps_1w_{11}w_{12}w_{31}w_{32}+2,\eps_2w_{21}w_{22}w_{31}w_{32})\Big(\frac{2}{w_{11}w_{12}w_{21}w_{22}}\Big)\Big)\Big(\frac{-1}{w_{11}w_{31}}\Big) \\
&\phantom{blahblah}+\ind_{\sigma\ne(-1,+1)}\cdot\frac{1}{2}\Big(1+\eta(\eps_1w_{11}w_{12}w_{31}w_{32},\eps_2w_{21}w_{22}w_{31}w_{32}+2)(\frac{2}{w_{11}w_{12}w_{21}w_{22}}\Big)\Big)\Big(\frac{-1}{w_{21}w_{31}}\Big) \\
&\phantom{blahblah}+\ind_{\sigma\ne(-1,-1)}\cdot\frac{1}{2}\Big(1+\eta(\eps_1w_{11}w_{12}w_{31}w_{32},\eps_2w_{21}w_{22}w_{31}w_{32})(\frac{2}{w_{11}w_{12}w_{21}w_{22}}\Big)\Big)\Big(\frac{-1}{w_{11}w_{21}}\Big)\bigg\},
\end{align*}
and the functions $r_h^\sigma(\vec{D};D_4)$ are bounded, integer-valued functions (with $|r_h^\sigma(\vec{D};D_4)|\le4$ for all tuples $\vec{D}$), and they are given by
\begin{align*}
r_1^\sigma(\vec{D};D_4)\;&:=\;\frac{1}{8}\Big(1+\eps_1\Big(\frac{-1}{D_{11}D_{12}D_{31}D_{32}}\Big)\Big)\Big(1+\eps_2\Big(\frac{-1}{D_{21}D_{22}D_{31}D_{32}}\Big)\Big) \\
&\phantom{blahblah}\times\Big(1+\eps_1\eps_2\Big(\frac{-1}{D_{11}D_{12}D_{21}D_{22}}\Big)\Big)\bigg\{-\ind_{\sigma\ne(\pm1,\mp1)}\psi(D_{11}D_{12}D_{21}D_{22}) \\
&\phantom{blahblahblah}-\ind_{\sigma\ne(\pm1,-1)}\psi(D_{21}D_{22}D_{31}D_{32})-\ind_{\sigma\ne(-1,\pm1)}\psi(D_{11}D_{12}D_{31}D_{32}) \\
&\phantom{blahblahblahblah}+\ind_{\sigma=(+1,+1)}\psi(D_{11}D_{12}D_{21}D_{22}D_{31}D_{32})\bigg\}, \\
r_2^\sigma(\vec{D};D_4)\;&:=\;\frac{1}{4}\Big(1+\eps_1\eps_2\Big(\frac{-1}{D_{11}D_{12}D_{21}D_{22}}\Big)\Big)\Big(1+\Big(\frac{2}{D_{11}D_{12}D_{21}D_{22}}\Big)\Big) \\
&\phantom{blahblah}\times\bigg\{-\ind_{\sigma\ne(\pm1,\mp1)}\psi(D_{11}D_{12}D_{21}D_{22})-\ind_{\sigma\ne(\pm1,-1)}\psi(D_{21}D_{22}D_{31}D_{32}) \\
&\phantom{blahblahblah}-\ind_{\sigma\ne(-1,\pm1)}\psi(D_{11}D_{12}D_{31}D_{32})+\ind_{\sigma=(+1,+1)}\psi(D_{11}D_{12}D_{21}D_{22}D_{31}D_{32})\bigg\},
\end{align*}
\begin{align*}
r_3^\sigma(\vec{D};D_4)\;&:=\;\frac{1}{8}\Big(1-\eps_1\Big(\frac{-1}{D_{11}D_{12}D_{31}D_{32}}\Big)\Big)\Big(1-\eps_2\Big(\frac{-1}{D_{21}D_{22}D_{31}D_{32}}\Big)\Big) \\
&\phantom{blahblah}\times\Big(1+\eps_1\eps_2\Big(\frac{-1}{D_{11}D_{12}D_{21}D_{22}}\Big)\Big)\bigg\{-\ind_{\sigma\ne(\pm1,\mp1)}\psi(D_{11}D_{12}D_{21}D_{22}) \\
&\phantom{blahblahblah}-\ind_{\sigma\ne(\pm1,-1)}\psi(D_{21}D_{22}D_{31}D_{32})-\ind_{\sigma\ne(-1,\pm1)}\psi(D_{11}D_{12}D_{31}D_{32}) \\
&\phantom{blahblahblahblah}+\ind_{\sigma=(+1,+1)}\psi(D_{11}D_{12}D_{21}D_{22}D_{31}D_{32})\bigg\}, \\
r_4^\sigma(\vec{D};D_4)\;&:=\;\frac{1}{4}\Big(1-\eps_1\eps_2\Big(\frac{-1}{D_{11}D_{12}D_{21}D_{22}}\Big)\Big)\Big(1+\Big(\frac{2}{D_{11}D_{12}D_{21}D_{22}}\Big)\Big) \\
&\phantom{blahblah}\times\bigg\{-\ind_{\sigma\ne(\pm1,\mp1)}\psi(D_{11}D_{12}D_{21}D_{22})-\ind_{\sigma\ne(\pm1,-1)}\psi(D_{21}D_{22}D_{31}D_{32}) \\
&\phantom{blahblahblah}-\ind_{\sigma\ne(-1,\pm1)}\psi(D_{11}D_{12}D_{31}D_{32})+\ind_{\sigma=(+1,+1)}\psi(D_{11}D_{12}D_{21}D_{22}D_{31}D_{32})\bigg\}.
\end{align*}

\end{prop}

\begin{proof}
  Fix $\sigma=(\eps_1,\eps_2)$, where $\eps_i=\sign d_i$. Proceeding as in the
  $Q_8$
  case above, we have
\begin{align*}
  B_1^\sigma(X;D_4)\;&=\;\frac{1}{m_1}\sum_{(d_1,d_2,d_3)\in \mathcal{S}_1^\sigma(X)}\ind_{D_4}(d_1,d_2,d_3)
  \\
&=\;\frac{1}{m_1}\flatsum_{\substack{\tau_1d_1d_2d_3\le X \\ \sign(d_1,d_2)=\sigma}}\frac{1}{2}\Big(1+\Big(\frac{-1}{d_1}\Big)\Big)\frac{1}{2}\Big(1+\Big(\frac{-1}{d_2}\Big)\Big)\frac{1}{2}\Big(1+\Big(\frac{-1}{d_3}\Big)\Big)\ind_{D_4}(d_1,d_2,d_3),
\end{align*}
We use the result of Lemma \ref{lem:d4-embedding-function} above (in the case where $2\nmid d_1d_2$), which gives
\begin{align*}
  B_1^\sigma(X;D_4)
  \;=\;
  \frac{1}{m_1}\flatsum_{\substack{\tau_1(D_1D_2D_3)^2\le X \\ \sigma=(\eps_1,\eps_2)}}&\frac{1}{8}\Big(1+\eps_1\Big(\frac{-1}{D_1D_3}\Big)\Big)\Big(1+\eps_2\Big(\frac{-1}{D_2D_3}\Big)\Big)\Big(1+\eps_1\eps_2\Big(\frac{-1}{D_1D_2}\Big)\Big) \\
&\times \sum_{D_i=D_{i1}D_{i2}} F_{\odd}(\vec{D};D_4)2^{-\omega(\vec{D})}g(\vec{D}). 
\end{align*}
As before, we interchange the order of summation. Unlike before, the function $F_{\odd}(\vec{D};D_4)$ does not only depend on the tuple $\vec{D}$ modulo 8. So we split up those terms from $F_{\odd}(\vec{D};D_4)$ that only depend on $\vec{D}$ modulo 8, and put the rest in a secondary summation. This gives
\[
B_1^\sigma(X;D_4)
\;=\;
\frac{1}{m_1}\sum_{\vec{w}\in H^6}f_1^\sigma(\vec{w};D_4)\flatsum_{\substack{\tau_1\Delta(\vec{D})\le X \\ \vec{D}\;\equiv\;\vec{w}\;(8) }}2^{-\omega(\vec{D})}g(\vec{D})\;+\flatsum_{\tau_1\Delta(\vec{D})\le X}r_1^\sigma(\vec{D};D_4)2^{-\omega(\vec{D})}g(\vec{D}),
\]
where $f_1^\sigma(\vec{w};D_4)$ and $r_1^\sigma(\vec{w};D_4)$ are the functions given in the statement of the proposition. Finally, as in the $Q_8$ case, we drop the conditions in the $\flatsum$ summations that require certain products of the $D_{ij}$ to not be 1: this makes them into $\natsum$ summations, at the cost of the $O(\sqrt{X})$ term. We handle the other
$
B_h^\sigma(X;D_4)
$
the same way, which completes the proof.
\end{proof}

\section{Character sums over linked variables}
\label{sec:linkvars}

Our results in the last section give uniform expressions for the quantities
$
B_h^\sigma(X;G)
$:
each one is a sum over a function of tuples $\vec{w}\in H^6$ times an inner character sum of the form
\begin{equation}
  G_{\vec{w}}(x;\Psi)
  \;:=\;
  \natsum_{
    \substack{
      \Delta(\vec{D})\le x
      \\
      \vec{D}\;\equiv\;\vec{w}\;(8)
    }
  }
  2^{-\omega(\vec{D})}
  \Psi(\vec{D})
  g(\vec{D}),
             \label{gw}
\end{equation}
whose dependence on $\vec{w}$ is much more mild. Here we are using the notation
\[
\Psi(\vec{D})\;:=\;\prod_{ij}\psi(D_{ij})^{\delta_{ij}}
\]
to denote a product of the functions $\psi(D_{ij})$ (for $1\le i\le3$, $1\le j\le 2$) where each $\delta_{ij}=0$ or
$1$
(and is fixed for the entire sum
    $
    G_w(x; \Psi)
    $).
    Specifically, $\Psi$ is the characteristic function for the set of tuples
    $
    \mathbf D
    $
    where those variables $D_{ij}$ with
    $
    \delta_{ij}=1
    $
    are divisible only by primes
    $
    \equiv\mymod{1}{4}
    $.
% In other words, recalling that
% $
% \psi
% $
% is the indicator function of an integer having all of its prime divisors $\equiv\mymod{1}{4}$, the function $\Psi$ simply restricts certain of the variables $D_{ij}$ (namely, those with $\delta_{ij}=1$) to have such prime divisors.

In the $Q_8$ case, no such restrictions appear---that is, $\delta_{ij}=0$ for all $i,j$, and the function $\Psi$ plays no role. The same is true for the first
terms of the expressions in the $D_4$ case. For the second terms in the $D_4$ case---those involving the functions $r_h^\sigma(\vec{D};D_4)$---one splits the summation among the four terms of the function $r_h^\sigma(\vec{D};D_4)$, and then each of those terms can be easily expressed in terms of $G_{\vec{w}}(x;\Psi)$ for different fixed choices of the $\delta_{ij}$. Therefore, after asymptotically evaluating the sum $G_{\vec{w}}(x;\Psi)$, all that is left to do is to evaluate the remaining finite sum over $\vec{w}\in H^6$; this is carried out in Section \ref{sec:asymptotic_evaluation_of_b_h_sigma_x_g_}.

Thus in this section we evaluate the sum $G_{\vec{w}}(x;\Psi)$; our result is the following.

\begin{prop}
\label{prop:gwx}
We have
\[
G_{\vec{w}}(x;\Psi)\;=\;\frac{\kappa_0}{4^3}\sqrt{x}(\log\sqrt{x})^{1/2}Y(\vec{w};\Psi)\;+\;O(\sqrt{x}(\log x)^{1/4}),
\]
where $\kappa_0>0$ is an absolute constant defined in {\rm (\ref{kappa-zero})}, and 
\[
Y(\vec{w};\Psi)\;=\;\ind_{w_{11}=w_{21}=w_{31}=1}\ind_{\delta_{12}=\delta_{22}=\delta_{32}=0}+\ind_{w_{12}=w_{22}=w_{32}=1}\ind_{\delta_{11}=\delta_{21}=\delta_{31}=0}.
\]
\end{prop}

We prove this proposition using the methods of Heath-Brown in \cite{heath}. This is an analysis involving ``linked'' variables $D_{ij},D_{k\ell}$, which are ones where the quadratic character $(D_{ij}/D_{k\ell})$ appears without its reciprocal in the summand under consideration, so that one expects its oscillations to manifest cancellations.
The precise statements of the intermediate lemmas involve a fair
        amount of notation; to help orient the reader, here is an outline
        of the argument.
First,
      after breaking the summation into dyadic intervals,
we
consider terms where there are two linked variables, each of which
must
be
(relatively) large (see~(\ref{whenever})).
These terms are shown to contribute negligibly to
$
G_{\vec{w}}(x,\Psi)
$
by an estimate (Lemma~\ref{lem:bilinear}) for bilinear forms involving the Kronecker symbol; this is the content of
Lemma~\ref{lem:linked}.
Next we consider the case of when there is one large variable, and all of the variables to which it is linked are (relatively) small.
These terms are also shown to contribute negligibly, this time by the Siegel-Walfisz theorem (Lemma~\ref{lem:charsum}); this is the content of Lemma~\ref{lem:onebig}.
Lastly, we show that the contribution of terms where many (i.e. four or more) variables are small is also negligible; this is the content of Lemma~\ref{lem:linked-var-summary}.
         It is worth mentioning that these lemmas are cumulative, i.e. each lemma builds on the previous ones and subsumes them.
After proving Lemma~\ref{lem:linked-var-summary}, we
         use it to finish the proof of Proposition~\ref{prop:gwx}---this also involves using Lemma~\ref{lem:triple-char}, which is itself an application of the Landau-Selberg-Delange method.
To keep the flow of the exposition here, we defer the proofs of
                  some of these technical lemmas (Lemmas~\ref{lem:bilinear}, \ref{lem:charsum}, and \ref{lem:triple-char}) to Section \ref{sec:lemma}.

Now we begin the proof of
Proposition~\ref{prop:gwx}.
Throughout the following $\vec{w}\in H^6$ and all $\delta_{ij}$ are fixed. First, we divide the range of each variable $D_{ij}$ into dyadic intervals $(A_{ij},2A_{ij}]$, where the $A_{ij}$ run through powers of 2 less than $\sqrt{x}$. This gives us $O(\log^6 x)$ nonempty subsums
\begin{equation}
\label{dyadic}
S(\vec{A})\;:=\;\natsum_{\substack{\Delta(\vec{D})\le x \\ \vec{D}\;\equiv\;\vec{w}\;(8) \\ A_{ij}<D_{ij}\le 2A_{ij} }}2^{-\omega(\vec{D})}\Psi(\vec{D})g(\vec{D}),\qquad \text{with}\qquad G_{\vec{w}}(x;\Psi)\;=\;\sum_{\vec{A}}S(\vec{A}),
\end{equation}
where we are writing $\vec{A}:=(A_{11},A_{12},A_{21},A_{22},A_{31},A_{32})$. In each subsum $S(\vec{A})$ we have $1\ll\prod_{ij}A_{ij}\ll \sqrt{x}$. Now recall that
\[
g(\vec{D})\;=\;\Big(\frac{D_{22}}{D_{11}}\Big)\Big(\frac{D_{32}}{D_{11}}\Big)\Big(\frac{D_{12}}{D_{21}}\Big)\Big(\frac{D_{32}}{D_{21}}\Big)\Big(\frac{D_{12}}{D_{31}}\Big)\Big(\frac{D_{22}}{D_{31}}\Big).
\]
We call the variables $D_{ij},D_{k\ell}$ ``linked'' if one of $(D_{ij}/D_{k\ell})$ or $(D_{k\ell}/D_{ij})$ appears as a factor of $g(\vec{D})$. Say that the latter appears in $g(\vec{D})$; we write $2^{-\omega(\vec{D})}\Psi(\vec{D})g(\vec{D})$ in the form
\[
2^{-\omega(\vec{D})}\Psi(\vec{D})g(\vec{D})\;=\;\Big(\frac{D_{k\ell}}{D_{ij}}\Big)a(D_{ij})b(D_{k\ell}),
\]
where the function $a(D_{ij})$ depends on all other variables $D_{uv}$ (including $D_{ij}$), but it is independent of $D_{k\ell}$, and similarly for the function $b(D_{k\ell})$. Note that $|a(D_{ij})|,|b(D_{k\ell})|\le1$. We can now write
\[
|S(\vec{A})|\;\le\;\sum_{1\le D_{uv}\le A_{uv}}\Big|\mathop{\sum_{D_{ij}}\sum_{D_{k\ell}}}_{D_{ij}D_{k\ell}\le \sqrt{x}}\Big(\frac{D_{k\ell}}{D_{ij}}\Big)a(D_{ij})b(D_{k\ell})\Big|,
\] 
and
we
apply
Lemma \ref{lem:bilinear} to the double summation over $D_{ij}$ and $D_{k\ell}$, which gives
\[
S(\vec{A})\;\ll\;\Big(\prod_{uv}A_{uv}\Big)A_{ij}A_{k\ell}\{\min(A_{ij},A_{k\ell})\}^{-1/32}\;\ll\;\sqrt{x}\;\{\min(A_{ij},A_{k\ell})\}^{-1/32}.
\]
Combining this estimate with the fact (stated just before \eqref{dyadic}) that there are $O(\log^6 x)$ subsums $S(\vec{A})$, we deduce the following.

\begin{lem}
\label{lem:linked}
We have
\[
S(\vec{A})\;\ll\;\sqrt{x}(\log x)^{-6}
\]
whenever there is a pair of linked variables with
\begin{equation}
A_{ij},A_{k\ell}\;\ge\;(\log x)^{192}.        \label{whenever}
\end{equation}
Therefore the total contribution of such subsums $S(\vec{A})$ to $G_\vec{w}(x;\Psi)$ is $\ll\sqrt{x}$.
\end{lem}

Now we consider those subsums
$
S(\vec{A})
$
where some
$
A_{ij}\ge(\log x)^{192}
$,
but
every variable $D_{uv}$ to which $D_{ij}$ is linked has
$
A_{uv}<(\log x)^{192}
$.
Letting $D$ be the product of these $D_{uv}$, we apply quadratic reciprocity as necessary to write $2^{-\omega(\vec{D})}\Psi(\vec{D})g(\vec{D})$ in the form
\[
2^{-\omega(\vec{D})}\Psi(\vec{D})g(\vec{D})\;=\;2^{-\omega(D_{ij})}\Big(\frac{D_{ij}}{D}\Big)\chi(D_{ij})\psi(D_{ij})^{\delta_{ij}} R,
\]
where $\chi$ is a Dirichlet character modulo 4, which may depend on the variables $D_{uv}$ other than $D_{ij}$, and the remaining factor $R$ is independent of $D_{ij}$ and satisfies $|R|\le1$. Recall that $\delta_{ij}=1$ or 0 depending on whether or not the term $\psi(D_{ij})$ appears in the product $\Psi(\vec{D})$.
Thus we have a bound for $S(\vec{A})$ of the form
\[
|S(\vec{A})|\;\le\;\sum_{D_{uv}}\Big|\sum_{D_{ij}}2^{-\omega(D_{ij})}\Big(\frac{D_{ij}}{D}\Big)\chi(D_{ij})\psi(D_{ij})^{\delta_{ij}}\Big|,
\]
where the inner summation runs over $D_{ij}$ that are odd, positive, squarefree, and coprime to all the other variables $D_{uv}$, and such that
\[
D_{ij}\equiv\mymod{w_{ij}}{8},\qquad A_{ij}<D_{ij}\le \min(2A_{ij},x'),
\]
where $x'$ depends on the variables $D_{uv}$ other than $D_{ij}$. We remove the condition $D_{ij}\equiv w_{ij}\;(8)$ using characters modulo
$8$,
and then we would like to apply Lemma \ref{lem:charsum} to bound the inner sum. The condition
  $
  A_{ij}\ge(\log x)^{192}
  $
  is not quite strong enough for Lemma \ref{lem:charsum} to be applied, however, and so we now instead require that
\[
    A_{ij}\;\ge\;z\;:=\;\exp\{(\log x)^\theta\},
\]
where
$
0<\theta<1
$
is a parameter to be
determined
   later.\footnote{Note that the parameter corresponding to $z$ in
   Heath-Brown's work \cite{heath} is
        ``$C=\exp\{\kappa(\log\log X)^2\}$.''
   Fouvry and Kl\"{u}ners \cite{fouvry-kluners} (who use the techniques of
   Heath-Brown to study the 4-rank of class groups of quadratic fields) have
   pointed out (see \cite[p.~460]{fouvry-kluners}) that this choice of
   parameter leads to an erroneous
   application of Heath-Brown's Lemma 6 (the analogue of our Lemma
   \ref{lem:charsum}) in \cite{heath}. This explains why we choose the
   parameter $z$ differently than in \cite{heath}. Fouvry--Kl\"{u}ners adapt
   their argument accordingly; our adaption is closer to the one in the work of
   Xiong--Zaharescu \cite{xiong}, who also use Heath-Brown's
   techniques.}
Additionally
   we define the parameter $N>0$ and take it large enough so that $N\theta>5\cdot192=960$. Now we apply Lemma \ref{lem:charsum}, taking (in the notation of that lemma) $r=\prod_{uv\ne ij}D_{uv}$, and taking
\begin{equation}
\label{sw-modulus-condition}
q\;=\;8D\;\ll\;(\log^{192}x)^5\;\ll\;(\log z)^{N},
\end{equation}
and we deduce that
\[
S(\vec{A})\;\ll\;A_{ij}\exp(-c\sqrt{\log A_{ij}})\sum_{r=\prod D_{uv}}\tau(r),
\]
provided that $D>1$. Here $\tau(r)$ denotes the usual divisor function, and $c=c_N$ is an absolute constant that depends on the parameter $N$.
Since the variables $D_{uv}$ are coprime in pairs, we have $\tau(r)=\prod_{uv}\tau(D_{uv})$. For a single variable $D_{uv}$ we have
\[
\sum_{D_{uv}}\tau(D_{uv})\;\ll\;A_{uv}\log A_{uv}\;\ll\;A_{uv}\log x,
\]
and hence we deduce
\[
S(\vec{A})\;\ll\;\sqrt{x}(\log x)^5\exp(-c\sqrt{\log z}).
\]
For $x$ sufficiently large, we have
\[
  \exp(-c\sqrt{\log z})\;=\;\exp(-c(\log x)^{\theta/2})\;\ll\;(\log x)^{-11}.
\]
With the following, we summarize the above result together with the result of Lemma \ref{lem:linked}.

\begin{lem}
\label{lem:onebig}
We have
\[
S(\vec{A})\;\ll\;\sqrt{x}(\log x)^{-6}
\]
whenever there are linked variables $D_{ij},D_{k\ell}$ for which
\[
A_{ij}\;\ge\;z\;=\;\exp\{(\log x)^\theta\}
\]
and
$
D_{k\ell}>1
$.
Therefore the total contribution of such subsums $S(\vec{A})$ to $G_\vec{w}(x;\Psi)$ is $\ll\sqrt{x}$.
\end{lem}

Next we estimate subsums $S(\vec{A})$ where at most two of the $D_{ij}$ have $A_{ij}\ge z$. Note that by taking $\theta$ slightly larger, we may assume that $z$ is a power of 2. If $\sum'$ denotes summation with the condition that at most two of the $A_{ij}\ge z$, then we have
\begin{equation}
\label{eq:few-large}
\primesum_{A_{ij}}|S(\vec{A})|\;\le\;\sum_{n_1\cdots n_6\le\sqrt{x}}2^{-\omega(n_1)}\cdots2^{-\omega(n_6)},
\end{equation}
where the $n_i$ are squarefree and pairwise coprime, and at most 2 of them are $\ge 2z$. Put
\[
m=\prod_{n_i<2z}n_i\qquad\text{and}\qquad n=\prod_{n_i\ge2z}n_i,
\]
so that $m<(2z)^6$ and $n\le \sqrt{x}/m$. Moreover, each value of $m$ can arise at most $6^{\omega(m)}$ times in \eqref{eq:few-large}, and each value of $n$ at most $\binom{6}{2}2^{\omega(n)}$ times. Therefore we have
\[
\primesum_{A_{ij}}|S(\vec{A})|\;\ll\;\sum_m\frac{6^{\omega(m)}}{2^{\omega(m)}}\sum_n\frac{2^{\omega(n)}}{2^{\omega(n)}}\;\ll\;\sqrt{x}\sum_{m\le(2z)^6}\frac{3^{\omega(m)}}{m}.
\]
For any fixed $\gamma>0$, we have the bound
\[
\sum_{m\le M}\gamma^{\omega(m)}\;\ll\;M(\log M)^{\gamma-1},
\]
which together with partial summation gives
\[
\sum_{m\le M}\frac{3^{\omega(m)}}{m}\;\ll\;(\log M)^3,
\]
and thus we have
\[
\primesum_{A_{ij}}|S(\vec{A})|\;\ll\;\sqrt{x}(\log x)^{3\theta}.
\]
Now choose the value $\theta=1/12$ (which also determines $N=12\cdot 960=11520$).
Combining the above estimate with the result of Lemma \ref{lem:onebig}, we may summarize as follows.

\begin{lem}
\label{lem:linked-var-summary}
We have
\[
\sum_{\vec{A}}|S(\vec{A})|\;\ll\;\sqrt{x}(\log x)^{1/4},
\]
where the sum is taken over all tuples $\vec{A}$ where there are at most two entries $A_{ij}\ge z$, or there are linked variables $D_{ij},D_{k\ell}$ with $A_{ij}\ge z$ and $D_{k\ell}>1$.
\end{lem}

Now with this lemma, we finish the proof of Proposition \ref{prop:gwx}.

\begin{proof}[Proof of Proposition \ref{prop:gwx}]
For a subsum $S(\vec{A})$ not to be eliminated by Lemma \ref{lem:linked-var-summary} (i.e., for a subsum to potentially contribute non-negligibly to $G_{\vec{w}}(x;\Psi)$), there must be at least three $A_{ij}\ge z$. These three entries must be either $A_{11},A_{21},A_{31}$ or $A_{12},A_{22},A_{32}$; if not (i.e., if it is some other subset of three entries), the shape of $g(\vec{D})$ is such that two of the corresponding variables $D_{ij},D_{k\ell}$ must be linked, and then this subsum would have been eliminated by Lemma \ref{lem:linked-var-summary}.
Moreover, if $A_{11},A_{21},A_{31}\ge z$, then we must have $D_{12}=D_{22}=D_{32}=1$ (and vice versa when $A_{12},A_{22},A_{32}\ge z$), since otherwise Lemma \ref{lem:linked-var-summary} would again have eliminated such a subsum. For each of these remaining terms we have $g(\vec{D})=1$. On recombining the dyadic subdivision of our original summation, we have now shown that
\[
G_{\vec{w}}(x;\Psi)\;=\;G_{\vec{w}}^1(x;\Psi)\;+\;G_{\vec{w}}^2(x;\Psi)\;+\; O(\sqrt{x}(\log x)^{1/4}),
\]
where we have put
\[
G_{\vec{w}}^1(x;\Psi)\;:=\natsum_{\substack{\Delta(\vec{D})\le x \\ \vec{D}\;\equiv\;\vec{w}\;(8) \\ D_{11}=D_{21}=D_{31}=1  }}2^{-\omega(\vec{D})}\Psi(\vec{D})\qquad \text{and}\qquad G_{\vec{w}}^2(x;\Psi)\;:=\natsum_{\substack{\Delta(\vec{D})\le x \\ \vec{D}\;\equiv\;\vec{w}\;(8) \\ D_{12}=D_{22}=D_{32}=1  }}2^{-\omega(\vec{D})}\Psi(\vec{D}).
\]
We evaluate $G_{\vec{w}}^1(x;\Psi)$ now; the same method applies analogously to $G_{\vec{w}}^2(x;\Psi)$. We see that $G_{\vec{w}}^1(x;\Psi)=0$ unless one has $w_{11}=w_{21}=w_{31}=1$. We use three sets of characters $\chi_1,\chi_2,\chi_3$ modulo 8 to detect the other three congruence conditions in $\vec{D}\equiv\vec{w}\;(8)$, which gives us
\begin{align*}
G_{\vec{w}}^1(x;\Psi)\;&=\; \prod_{1\le i\le3}\ind_{w_{i1}=1}\cdot\frac{1}{4^3}\sum_{\chi_1}\sum_{\chi_2}\sum_{\chi_3}\overline{\chi_1}(w_{12})\overline{\chi_2}(w_{22})\overline{\chi_3}(w_{32}) \\
&\phantom{blah}\times \natsum_{(D_{12}D_{22}D_{32})^2\le x}2^{-\omega(D_{12}D_{22}D_{32})}\chi_1(D_{12})\chi_2(D_{22})\chi_3(D_{32})\psi(D_{12})^{\delta_{12}}\psi(D_{22})^{\delta_{22}}\psi(D_{32})^{\delta_{32}},
\end{align*}
where $\ind_{w_{i1}=1}=1$ if and only if $w_{i1}=1$, else it is $=0$. We now apply Lemma \ref{lem:triple-char} to evaluate the inner sum, and we get
\[
  G_{\vec{w}}^1(x;\Psi)
  \;=\;
  \frac{\kappa_0}
       {4^3}
  \sqrt{x}(\log \sqrt{x})^{J}
  \prod_{1\le i\le3}
     \ind_{w_{i1}=1}
     \cdot
     (1+\delta_{i2}\chi_{-1}(w_{i2}))
  \;+\;
     O(\sqrt{x}(\log x)^{J-1}),
\]
where $\kappa_0$ is the constant defined in \eqref{kappa-zero}, and
$
J=(2-\delta_{12}-\delta_{22}-\delta_{32})/4
$.
Observe that if at least one $\delta_{i2}=1$, then
$
J\le 1/4
$,
and hence we have
\[
G_{\vec{w}}^1(x;\Psi)\;=\;\frac{\kappa_0}{4^3}\sqrt{x}(\log \sqrt{x})^{1/2}\ind_{w_{11}=w_{21}=w_{31}=1}\ind_{\delta_{12}=\delta_{22}=\delta_{32}=0}\;+\;O(\sqrt{x}(\log x)^{1/4}).
\]
Evaluating $G_{\vec{w}}^2(x;\Psi)$ the same way and then combining the two completes the proof.
\end{proof}

% \begin{remm}
%   If all $\delta_{ij}=0$, we can replace the error term in the proposition with the slightly smaller one
%   $
%   O(\sqrt{x}(\log\log x)^6)
%   $,
%   though this makes no difference for the   purpose of proving Theorem
%   \ref{thm:density}.
% \end{remm}

\section{Asymptotic evaluation of $B_h^\sigma(X;G)$}
\label{sec:asymptotic_evaluation_of_b_h_sigma_x_g_}

In this section we use the results of Sections \ref{sec:average} and \ref{sec:linkvars} to prove the following.

\begin{thm}
\label{thm:main-thm}
Let $\kappa_0$ be the constant defined in {\rm (\ref{kappa-zero})}.
For both $G=Q_8$ and $D_4$, we have
\[
  B^\sigma(X;G)\;=\;K^\sigma(G)\sqrt{X}(\log X)^{1/2}\;+\;O(\sqrt{X}(\log X)^{1/4}),
\]
where the constants $K^\sigma(G)$ are as follow:
\begin{itemize}
  \item for $G=Q_8$, $K^\sigma(Q_8)=0$ unless $\sigma=(+1,+1)$, and
  \[
    K^{(+1,+1)}(Q_8)\;=\;\frac{25 \kappa_0}{192\sqrt{2}}
  \]
  \item for $G=D_4$ we have
\[
    K^{(+1,+1)}(D_4)=\frac{33\kappa_0}{64\sqrt{2}},
    \quad
    K^{(+1,-1)}(D_4)
    =
    K^{(-1,+1)}(D_4)=\frac{31\kappa_0}{96\sqrt{2}},
    \quad
    K^{(-1,-1)}(D_4)=\frac{37\kappa_0}{96\sqrt{2}}.
\]
\end{itemize}
\end{thm}

This theorem is essentially the same as Theorem \ref{thm:density}, the only difference being that in Theorem \ref{thm:density} we have grouped together the terms that contribute to the totally real and totally imaginary cases, and we have rewritten the constants. Note also our slight ``abuse'' of notation: here $\sigma$ denotes a pair of signs, whereas $\sigma\in\{\pm\}$ in the statement of Theorem \ref{thm:density} denotes whether the fields in question are totally real or totally imaginary. 

\begin{proof}[Proof of Theorem \ref{thm:main-thm}]
In Section \ref{sec:average} we showed that
\begin{align}
\label{bh-for-d4}
B_h^\sigma(X;D_4)\;&=\;\frac{1}{m_h}\sum_{\vec{w}\in H^6}f_h^\sigma(\vec{w};D_4)\natsum_{\substack{\Delta(\vec{D})\le X/T_h \\ \vec{D}\;\equiv\;\vec{w}\;(8) }}2^{-\omega(\vec{D})}g(\vec{D})\\
&\phantom{blahblah}+\natsum_{\Delta(\vec{D})\le X/T_h}r_h^\sigma(\vec{D};D_4)2^{-\omega(\vec{D})}g(\vec{D})\;+\;O(\sqrt{X}). \nonumber
\end{align}
Let $S$ be the last summation above, i.e., the one whose summand includes the function
$
r_h^\sigma(\vec{D};D_4)
$.
By splitting the summation over $\vec{D}$ into progressions
$
\vec{D}\equiv\vec{w}\;(8)
$,
we can write $S$ as a sum of several (four, in fact) sums of the form
\[
\sum_{\vec{w}\in H^6}s(\vec{w})\natsum_{\substack{\Delta(\vec{D})\le X/T_h \\ \vec{D}\;\equiv\;\vec{w}\;(8) }}2^{-\omega(\vec{D})}\Psi(\vec{D})g(\vec{D}),
\]
where each $s$ is function on $H^6$ with values in $\{+1,0,-1\}$, and $\Psi(\vec{D})=\prod_{ij}\psi(D_{ij})^{\delta_{ij}}$. The function $\Psi$ is different in each of the four summations, but in each case it has at least four $\delta_{ij}=1$. By Proposition \ref{prop:gwx}, then, each of these sums is $O(\sqrt{X}(\log X)^{1/4})$, and hence the contribution of $S$ to \eqref{bh-for-d4} is negligible. Thus we now have, for both $G=Q_8$ and $D_4$, 
\[
 B_h^\sigma(X;G)\;=\;\frac{1}{m_h}\sum_{\vec{w}\in H^6}f_h^\sigma(\vec{w};G)\natsum_{\substack{\Delta(\vec{D})\le X/T_h \\ \vec{D}\;\equiv\;\vec{w}\;(8) }}2^{-\omega(\vec{D})}g(\vec{D})\;+\;O(\sqrt{X}(\log X)^{1/4}).
\]
(Note that in the statement of Proposition \ref{prop:fhq8} we defined the functions $f_h(\vec{w};Q_8)$ without the superscript $\sigma$; here we write them with the superscript to be able to write a single expression for the $Q_8$ and $D_4$ cases, with the understanding that $f_h^\sigma(\vec{w};Q_8):=0$ whenever $\sigma\ne(+1,+1)$.) We apply the result of Proposition \ref{prop:gwx} to the $\natsum$ summation, which gives us
\[
B_h^\sigma(X;G)\;=\;K_h^\sigma(G)\sqrt{X}(\log X)^{1/2}\;+\;O(\sqrt{X}(\log X)^{1/4}),
\]
where
\begin{equation}
\label{c-h-sigma-g}
K_h^\sigma(G)\;:=\;\frac{\kappa_0S_h^\sigma(G)}{4^3m_h\sqrt{2T_h}},
\end{equation}
and where we have put
\[
  S_h^\sigma(G)\;:=\sum_{\vec{w}\in H^6}f_h^\sigma(\vec{w};G)Y(\vec{w}),
\]
with
\[
Y(\vec{w})\;:=\;\ind_{w_{11}=w_{21}=w_{31}=1}+\ind_{w_{12}=w_{22}=w_{32}=1}.
\]
It remains to evaluate the sums $S_h^\sigma(G)$, which is simply done with a computer, for instance. We record the results of those computations here. For $G=Q_8$, we have $f_h^\sigma(\vec{w},Q_8)=0$ unless $\sigma=(+1,+1)$, and in this case we compute the following values:
\begin{center}
  \begin{tabular}{c | c c c c}
    $
    \sigma
    $
    &
    $S_1^\sigma(Q_8)$ & $S_2^\sigma(Q_8)$ & $S_3^\sigma(Q_8)$ & $S_4^\sigma(Q_8)$
    \\
    \hline
    $(+1,+1)$
    &
    32 & 32 & 0 & 32
  \end{tabular}
\end{center}
For $G=D_4$, we compute the following values:
\begin{center}
  \begin{tabular}{c | c c c c}
    $
    \sigma
    $
    &
    $S_1^\sigma(D_4)$ & $S_2^\sigma(D_4)$ & $S_3^\sigma(D_4)$ & $S_4^\sigma(D_4)$ \\
    \hline
    $(+1,+1)$
    &
    96 & 96 & 64 & 96
    \\
    $
    (+1,-1)
    $
    & 64 & 64 & 32 & 64
    \\
    $
    (-1,+1)
    $
    & 64 & 64 & 32 & 64
    \\
    $
    (-1,-1)
    $
    &
    64 & 64 & 64 & 64 
  \end{tabular}
\end{center}
We use these values to compute
$
K_h^\sigma(G)
$
via \eqref{c-h-sigma-g}.  Recalling \eqref{sum-over-h}, we have that
$
K^\sigma(G)=\sum_{h=1}^4K_h^\sigma(G)
$;
summing over $1\le h\le 4$ then gives the result. 
\end{proof}

\section{Counting biquadratic fields}

As another application of our methods, we give a proof of the asymptotic count of biquadratic fields (up to isomorphism) of bounded discriminant. As we have mentioned in the introduction, this result is originally due to Baily \cite{baily} (and corrected by M\"{a}ki \cite{maki}). Nevertheless we provide the proof here as an alternative method to derive this asymptotic.

Recall that $B(X)$ denotes the number of biquadratic fields (up to isomorphism) of discriminant at most $X$, and that $B^+(X)$ and $B^-(X)$ denote the numbers of such fields that are totally real and totally imaginary, respectively. 

\begin{thm}
For $\sigma\in\{\pm\}$, we have
\[
    B^\sigma(X)
    \;=\;
    \frac{23     c^\sigma}
         {960}
    \sqrt{X} (\log X)^2
    \prod_p
      \Bigl(         
         \Bigl(        
           1 - \frac{1}{p}
         \Bigl)^3
         \Bigl(        
           1 + \frac{3}{p}
         \Bigl)
      \Bigl)
    \;+\;
    O(\sqrt{X} \log X),    
\]
where $c^+=1/4$ and $c^-=3/4$.
\end{thm}

\begin{proof}
  We count the fields by their type $h$; i.e., we evaluate each sum
  \[
    B_h^\sigma(X)\;=\;\frac{1}{m_h}\sum_{(d_1,d_2,d_3)\in \mathcal{S}_h^\sigma(X)}1.
  \]
  We follow the same methods as in Section \ref{sec:average}. This time we write $\vec{D}=(D_1,D_2,D_3)$ for a triple of integers. They will always be positive, odd, squarefree, and pairwise coprime, so we use the notation $\natsum_{\vec{D}}$ to denote a sum over such triples. The notation $\Delta(\vec{D})$ now denotes
  \[
    \Delta(\vec{D})\;=\;(D_1D_2D_3)^2.
  \]
  We will consider these triples modulo 4 (as opposed to modulo 8 before), so we now put $H=(\Z/4\Z)^\times$, and we write triples $\vec{w}\in H^3$ as 
  \[
    \vec{w}\;=\;(w_1,w_2,w_3).
  \]
  With this notation, we have
  \begin{equation}
  \label{bh-for-biquads}
    B_h^\sigma(X)\;=\;\frac{1}{m_h}\sum_{\vec{w}\in H^3}f_h^\sigma(\vec{w})\natsum_{\substack{\Delta(\vec{D})\le X/T_h \\ \vec{D}\;\equiv\;\vec{w}\;(4)}}1\;+\;O(\sqrt{X}),
  \end{equation}
  where
  \begin{align*}
  f_1^\sigma(\vec{w})\;&=\;\frac{1}{8}\Big(1+\eps_1\Big(\frac{-1}{w_1w_3}\Big)\Big)\Big(1+\eps_2\Big(\frac{-1}{w_2w_3}\Big)\Big)\Big(1+\eps_1\eps_2\Big(\frac{-1}{w_1w_2}\Big)\Big), \\
  f_2^\sigma(\vec{w})\;&=\;\frac{1}{2}\Big(1+\eps_1\eps_2\Big(\frac{-1}{w_1w_2}\Big)\Big), \\
  f_3^\sigma(\vec{w})\;&=\;\frac{1}{8}\Big(1-\eps_1\Big(\frac{-1}{w_1w_3}\Big)\Big)\Big(1-\eps_2\Big(\frac{-1}{w_2w_3}\Big)\Big)\Big(1+\eps_1\eps_2\Big(\frac{-1}{w_1w_2}\Big)\Big), \\
  f_4^\sigma(\vec{w})\;&=\;\frac{1}{2}\Big(1-\eps_1\eps_2\Big(\frac{-1}{w_1w_2}\Big)\Big).
  \end{align*}
  We evaluate the sum over $\vec{D}$ by using Dirichlet characters modulo 4 to detect the congruences and then applying Lemma \ref{lem:triple-char2}, which gives
  \[
    \natsum_{\substack{\Delta(\vec{D})\le X/T_h \\ \vec{D}\;\equiv\;\vec{w}\;(4)}}1\;=\;\frac{\lambda_0}{2^5}\sqrt{\frac{X}{T_h}}(\log X)^2\;+\;O(\sqrt{X}\log X),
  \]
  where $\lambda_0$ is the constant given by \eqref{lambda0}. Plugging this into \eqref{bh-for-biquads}, we get
  \[
  B_h^\sigma(X)\;=\;\frac{\lambda_0S_h^\sigma}{2^5m_h\sqrt{T_h}}\sqrt{X}(\log X)^2\;+\;O(\sqrt{X}\log X),
  \]
  where
  \[
    S_h^\sigma\;=\;\sum_{\vec{w}\in H^3}f_h^\sigma(\vec{w}).
  \]
  In fact, the sums $S_h^\sigma$ do not depend on $\sigma$: for any pair of signs $\sigma=(\eps_1,\eps_2)$, we calculate
  \[
    S_1^\sigma\;=\;S_3^\sigma\;=\;2\quad\text{and}\quad S_2^\sigma\;=\;S_4^\sigma\;=\;4.
  \]
  Summing over $1\le h\le 4$ now gives, for any $\sigma$, 
  \[
    B^\sigma(X)\;=\;\frac{23}{3840}\prod_p\Big(\Big(1-\frac{1}{p}\Big)^3\Big(1+\frac{3}{p}\Big)\Big)\cdot\sqrt{X}(\log X)^2\;+\;O(\sqrt{X}\log X).
  \]
  This asymptotic gives the result, after summing over the relevant $\sigma$. (Note again our slight abuse of notation: in the proof $\sigma$ refers to a pair of signs, whereas in the statement of the theorem, $\sigma\in\{\pm\}$ denotes whether the fields in question are totally real or totally imaginary). 
\end{proof}

\section{Lemmas on character sums}
\label{sec:lemma}

In this section we record a few estimates on certain character sums and other sums of multiplicative functions. The first of these is a bound for bilinear forms involving the Jacobi symbol. The statement is almost exactly the same as Lemma 4 in \cite{heath} (though with slightly different conditions modulo 8), and the proof given there gives this result as well.

\begin{lem}
\label{lem:bilinear}
Let $a_m,b_n$ be any complex numbers with $|a_m|,|b_n|\le1$. Let $h,k$ be given odd integers, and let $M,N,x\gg1$. Then
\[
\sum_{m}\sum_{n}\Big(\frac{n}{m}\Big)a_mb_n\;\ll\;MN\{\min(M,N)\}^{-1/32},
\]
uniformly in $x$, where the summation runs over squarefree $m,n$ satisfying $M<m\le2M$, $N<n\le2N$, $mn\le x$, $m\equiv h\;(8)$, and $n\equiv k\;(8)$. 
\end{lem}

The
following
lemma is almost identical to Lemma 6 in \cite{heath} and the proof is completely analogous to the one given there. In the case that $\delta=0$, the proof is effectively identical after changing $4^{-\omega(n)}$ there to $2^{-\omega(n)}$ (and making the other corresponding changes as necessary). In the case that $\delta=1$, the same proof works, replacing $L(s,\chi)$ there with $L(s,\chi)L(s,\chi\chi_{-1})$ and modifying the holomorphic factor $g(s)$ as necessary.

\begin{lem}
\label{lem:charsum}
Fix $N>0$, and let $\delta=0$ or $1$. Let $q,r>0$ be positive integers, and let $\chi$ be a nonprincipal character modulo $q$ that is not induced by $\chi_{-1}$, the nonprincipal character modulo $4$. Then we have
\[
\sum_{\substack{n\le x \\ (n,r)=1}}\mu^2(n)2^{-\omega(n)}\chi(n)\psi(n)^\delta\;\ll\;x\tau(r)\exp(-c\sqrt{\log x})
\]
where $c=c_N$ is a positive constant, and this bound holds uniformly for $q\le(\log x)^N$. The implied constant above is ineffective.
\end{lem}

% Next we give some estimates and asymptotic formulas for sums of the form 
% \[
% \sum_{n\le x}f(n)
% \]
% for certain multiplicative functions $f$. We write $F(s)=\sum_{n\ge1}f(n)n^{-s}$ for the Dirichlet series associated to $f$, and we assume that $F(s)^k=G(s)L(s)$, where $k\ge1$ is an integer, $G(s)$ is holomorphic in the region
% \[
% \sigma\;\ge\;1-\frac{c_1}{1+\log^+|t|}.
% \]

% \begin{proof}
% When $L(s)$ is entire, we follow the proof of Heath-Brown in \cite{heath}; we briefly sketch the approach here. The function $F(s)$ has analytic continuation into any region $\sigma\ge\sigma_0>1/2$, $|t|\le T$ that is free of zeros of $L(s)$. Assuming $q\le(\log x)^N$ and taking $T=\exp(\sqrt{\log x})$, Siegel's theorem gives a zero-free region $R$ for the product $L(s)$, 
% \[
% R\;=\;\Big\{s:\sigma\ge1-\frac{c}{\log T},|t|\le T\Big\},
% \]
% where the constant $c=c(N)$ depends on $N$ and is ineffective. For each factor $L(s,\chi)$ of $L(s)$, we have $L(s,\chi)\ll\log T$ for $s\in R$. Thus $F(s)\ll M(\log x)^A$ for $s\in R$. Applying Perron's formula and moving the contour left to $\sigma=1-c/\log T$, one concludes from standard arguments that
% \end{proof}

    The
  rest of the lemmas in this section have no
  direct counterpart in \cite{heath} so
we give their proofs for completeness' sake.

\begin{lem}
\label{sqfree-divisor-bound}
We have
\begin{equation}
\label{eq:sqf-div-bd}
\sum_{n\le x}\mu^2(n)2^{-\omega(n)}\tau(n)\;\ll\;x.
\end{equation}
\end{lem}

\begin{proof}
We consider the Dirichlet series
\[
F(s)\;=\;\sum_{n\ge1}\mu^2(n)2^{-\omega(n)}\tau(n)n^{-s}\;=\;\prod_p(1+p^{-s}),
\]
for which we have
\[
F(s)\;=\;\zeta(s)G(s),\quad\text{where}\quad G(s)\;=\;\prod_p(1-p^{-2s}).
\]
Because $G$ is holomorphic for $\Re(s)>1/2$, a standard application of Perron's formula gives the bound \eqref{eq:sqf-div-bd}.
\end{proof}

\begin{lem}
\label{lem:triple-char}
Let $\chi_1,\chi_2,\chi_3$ be three Dirichlet characters modulo 8, and put
\[
\rho(n)\;:=\;\sum_{n=abc}\chi_1(a)\chi_2(b)\chi_3(c)\psi(a)^{\delta_1}\psi(b)^{\delta_2}\psi(c)^{\delta_3},
\]
where each $\delta_i=0$ or $1$. Define
\[
\Upsilon\;:=\;\prod_{1\le i\le3}(\ind_{\chi_i=\chi_0}+\delta_i\ind_{\chi_i=\chi_{-1}}),
\]
and put $J:=(2-\delta_1-\delta_2-\delta_3)/4$. Then we have
% \[
% \sum_{n\le x}\mu^2(n)\rho(n)\;=\;\lambda_0x(\log x)^2\ind_{\chi_i=\chi_0}+O(x\log x)
% \]
% and
\begin{equation}
\label{lsd-app}
\sum_{n\le x}\mu^2(n)2^{-\omega(n)}\rho(n)\;=\;\kappa_0(\delta_1,\delta_2,\delta_3)\cdot x(\log x)^{J}\cdot\Upsilon+O(x(\log x)^{J-1}),
\end{equation}
where $\kappa_0(\delta_1,\delta_2,\delta_3)>0$ is a constant that only depends on the $\delta_i$. We will only need its value when
$
\delta_1=\delta_2=\delta_3=0
$,
which is
\begin{equation}
    \kappa_0 \;:=\;\kappa_0(0,0,0)
    \;=\;
    \frac{8}{7\sqrt{\pi}}
    \prod_p
      \Bigl(
        \Bigl(
          1 - \frac{1}{p}
        \Bigr)^{3/2}
        \Bigl(
          1 + \frac{3}{2p}
        \Bigr)
      \Bigr).                  \label{kappa-zero}
\end{equation}
Note also that in this case we have
$
\Upsilon=\prod_{i}\ind_{\chi_i=\chi_0}
$.
\end{lem}

\begin{proof}
  For $i=1,2,3$, define
  \[
    L_{\delta_i}(s,\chi_i)\;=\;\prod_p(1-\chi_i(p)\psi(p)^{\delta_i}p^{-s})^{-1};
  \]
  note that the product is over primes $p\ge3$ because $\chi_i$ is a character modulo 8. We have
  \[
    L_{\delta_i}(s,\chi_i)^2\;=\;L(s,\chi_i)L(s,\chi_i\chi_{-1}^{\delta_i})g(s)^{-\delta_i},
  \]
  where
  \[
    g(s)\;=\;\prod_{p\equiv3\;(4)}(1-\chi_i(p)^2p^{-2s})^{-1}\;=\;\prod_{p\equiv3\;(4)}(1-p^{-2s})^{-1}.
  \]
  Note that in this last equality we are using the fact that the $\chi_i$ are characters modulo 8, so for each $i$ we have $\chi_i^2=\chi_0$, the principal character modulo 8. We consider the series
  \[
    F(s)\;=\;\sum_{n\ge1}\mu^2(n)2^{-\omega(n)}\rho(n)n^{-s}\;=\;\prod_{p\ge3}\Big(1+\tfrac{1}{2}(\chi_1(p)\psi(p)^{\delta_1}+\chi_2(p)\psi(p)^{\delta_2}+\chi_3(p)\psi(p)^{\delta_3})p^{-s}\Big),
  \]
  and we see that
  \begin{equation}
  \label{f-g-lschi}
    F(s)^4=G(s)\prod_{1\le i\le3}L_{\delta_i}(s,\chi_i)^2,
  \end{equation}
  where $G(s)$ is holomorphic for $\Re(s)>1/2$. We now apply the Landau-Selberg-Delange method (see Chapter II.5 in \cite{tenenbaum}, for instance), which gives \eqref{lsd-app}. (Note that the term $\Upsilon$ indicates the conditions under when poles at $s=1$ appear on the right-hand side of \eqref{f-g-lschi}, and
  $
  4(J+1)
  $ is precisely the number of such poles.)

  To compute the constant $\kappa_0=\kappa_0(0,0,0)$, note that
  \[
    L_0(s,\chi_i)\;=\;L(s,\chi_0)\;=\;\prod_{p\ge3}(1-p^{-s})^{-1},
  \]
  and hence when $\delta_1=\delta_2=\delta_3=0$ we have $F(s)\;=\;(1-2^{-s})^{3/2}\zeta(s)^{3/2}G(s)^{1/4}$, where
  \[
    G(s)\;=\;\prod_{p\ge3}\Big(\Big(1-\frac{1}{p^s}\Big)^6\Big(1+\frac{3}{2p^s}\Big)^4\Big).
  \]
  It then follows from the results in Chapter II.5 in \cite{tenenbaum} that
  \[
    \kappa_0\;=\;\Gamma(\tfrac{3}{2})^{-1}\Big[\frac{(s-1)^{3/2}F(s)}{s}\Big]_{s=1}\;=\;\frac{2}{\sqrt{\pi}}\cdot\frac{G(1)^{1/4}}{2^{3/2}},
  \]
  and we see that
  $
  \kappa_0(0,0,0)
  $
  is equal to $\kappa_0$ as defined in (\ref{kappa-zero}).
\end{proof}

\begin{lem}
\label{lem:triple-char2}
Let $\chi_1,\chi_2,\chi_3$ be three Dirichlet characters modulo 4, and put
\[
\rho_1(n)\;:=\;\sum_{n=abc}\chi_1(a)\chi_2(b)\chi_3(c),
\]
and
\[
  \Upsilon_1
\;:=\;
\prod_{1\le i\le3}\ind_{\chi_i=\chi_0}.
\]
Then we have
\[
\sum_{n\le x}\mu^2(n)\rho_1(n)
\;=\;
\lambda_0x(\log x)^{2}\cdot \Upsilon_1\;+\;O(x\log x),
\]
where the constant $\lambda_0>0$ is given by
\begin{equation}
\label{lambda0}
  \lambda_0\:=\;\frac{1}{5}\prod_{p}\Big(\Big(1-\frac{1}{p}\Big)^3\Big(1+\frac{3}{p}\Big)\Big).
\end{equation}
\end{lem}

\begin{proof}
  The proof is similar to the previous one, but simpler. We consider the Dirichlet series
  \[
  F(s)
  \;=\;
  \sum_{n\ge1}\mu^2(n) \rho_1(n) n^{-s}
  \;=\;
  \prod_{p\ge3}\Big(1+(\chi_1(p)+\chi_2(p)+\chi_3(p))p^{-s}\Big),
  \]
  and we see that
  \[
    F(s)\;=\;G(s)\prod_{1\le i\le 3}L(s,\chi_i),
  \]
  where $G(s)$ is holomorphic for $\Re(s)>1/2$. If not all of the $\chi_i=\chi_0$, then the right-hand side above has at most 2 poles at $s=1$, and an application of Perron's formula shows that
  \[
     \sum_{n\le x}\mu^{2}(n) \rho_1(n)
     \;\ll\;
     x\log x. 
  \]
  Otherwise the conditions indicated by the term $\Upsilon$ are met (in other words, all three $\chi_i=\chi_0$) and $F(s)=(1-2^{-s})^3\zeta(s)^3G(s)$, where
  \[
    G(s)\;=\;\prod_{p\ge3}\Big(\Big(1-\frac{1}{p^s}\Big)^3\Big(1+\frac{3}{p^s}\Big)\Big).
  \]
  Again, it follows from the results in Chapter II.5 of \cite{tenenbaum} that
  \[
    \sum_{n\le x}\mu^2(n) \rho_1(n)
    \;=\;
    \lambda_0x(\log x)^2\;+\;O(x\log x),
  \]
  where
  \[
    \lambda_0\;=\;\Gamma(3)^{-1}\Big[\frac{(s-1)^{3}F(s)}{s}\Big]\;=\;\frac{1}{2}\cdot\frac{G(1)}{2^3},
  \]
  which is the value of $\lambda_0$ given in \eqref{lambda0}.
\end{proof}

\section*{Appendix: Connection with Galois representations}

\renewcommand\thesubsection{A.\arabic{subsection}}

In this appendix we revisit our $Q_8$- and $D_4$-extension problems from the
point of view of Galois representations.
This discussion is not needed for and is of a different nature from the rest
of the paper, but we hope
it will help the readers better understand the motivation of our results;
it also suggests directions of further works
  and allows one to better understand the relationship between the works of Rohrlich
  \cite{rohrlich-dih-type, rohrlich-quat, rohrlich-quat2, rohrlich-dihedral}
  and this paper.

\subsection{Projective Artin representations and lifts}

Fix an algebraic closure $\ov{\mathbb Q}$ of $\Q$, and set
$
G_\Q = \gal(\ov{\mathbb Q}/\Q)
$.
Let
$
\trho: G_\Q \to PGL_2(\CC)
$
be a projective Artin representation.  Tate \cite[Theorem 4]{serre:one} proved
that any such $\tilde{\rho}$ can be lifted to a representation
$
\rho: G_\Q\to GL_n(\CC)
$,
in other words
$
\trho = \pi \, \text{\footnotesize$\circ$} \, \rho
$
where
$
\pi
$
is the quotient map
$
GL_2(\CC) \to PGL_2(\CC)
$.
Then
$
\text{image}(\rho)
$
is the central extension of
$
\text{image}(\trho)
$
by a finite cyclic group
$
C(\rho)
$:

\begin{center}
  \begin{tikzcd}
    &
    G_{\Q}
      \arrow[d, dashed,"\rho"]
      \arrow[dr, "\trho"']
    \\
    C(\rho)
      \arrow[hook, r]  \arrow[hook, d]
    &
    \text{image}(\rho)
      \arrow[r, two heads]  \arrow[hook, d]
    &
    \text{image}(\trho)  \arrow[hook, d]
    \\
    \CC^\times
      \arrow[hook, r]
    &
    GL_2(\CC)
      \arrow[r, two heads, "\pi"]
    &
    PGL_2(\CC)
  \end{tikzcd}
\end{center}

\noindent
We call
$
\#C(\rho)
$
the index of the lift
$
\rho
$.

From now on, suppose
$
\trho
$
is irreducible, i.e.~suppose it admits an irreducible lift.  This notion is
well-defined, since any two lifts differ by the twist of a linear character
$
G_\Q \to \CC^\times
$.
Then
$
\text{image}(\trho)
$
is either dihedral, $A_4$, $S_4$,  or $A_5$.
      %
      %
      % in particular
      % $
      % \#G(\trho)
      % $
      % is always even.   Schur-Zassenhaus then implies that
      % $
      % n(\rho)
      % $
      % must be even.
      %
      %
Among the irreducible
$
\trho
$,
the smallest possible image for 
$
\trho
$
is $V_4$, the Klein four-group, and for such 
$
\trho
$
the smallest possible index is $2$: the quaternion group $Q_8$ and the dihedral group
$
D_4
$
of order $8$ both admit an irreducible
representation into
$
GL_2(\CC)
$
with $V_4$ projective image.

Now let
$K$
be a biquadratic number field, viewed as the splitting field of an
irreducible
$
\tilde{\rho}_K: G_\Q\to PGL_2(\CC)
$.
Tate's theorem guarantees that
$
\tilde{\rho}_K
$
has a lift, but there need not be a lift of a given index, in particular there
need not be a lift of index
$2$.
In concrete terms, that means $K$ need not be a subfield of a $Q_8$- or $D_4$-extension of
$\Q$.
There are also biquadratic number fields that can be extended into $Q_8$- as
well as $D_4$-extensions.  When this happens, we can twist the corresponding
$Q_8$-Galois representations into $D_4$-ones, and vice
versa.
  %\marginpar{{\footnotesize
  %  2nd: density of fields with given index}}
  %
Theorem \ref{thm:density} now tells us that almost all biquadratic fields, as ordered
by field discriminants, do not admit $Q_8$- or $D_4$-lifts.
%
% The goal of this paper is to study the density of biquadratic number fields
% for which these embedding problems have a solution.

\subsection{Future works}
    \label{sec:future}

The Galois representation point of view of our extension problem naturally
gives rise to a number of questions.  To streamline the discussion,
given a biquadratic field $K$ we define
\[
   R_K \;:=\; \{  \text{index}(\rho):   \text{$\rho$ is a lift of $\trho_K$} \},
\]
the set of indices of lifts of $\trho_K$.  We can show that  $R_K\subseteq 2\mathbb N$, the set of
positive even integers.
We also define
\[
   \iota(K)
   \;:=\;
   \min R_K
\]
which we call the \textit{minimal index} of $K$.

\begin{quest}
  \label{quest:rk}
  (density of realizable indices)
  
For a given biquadratic field $K$, is the set
$
2\mathbb N - R_K
$
finite?  In other words, does there exist a constant $c_1(K)>0$ depending only
on $K$ such that every even integer
$
\ge c_1(K)
$
is the index of a lift of $\trho_K$?
\end{quest}

\begin{quest}
  \label{quest:cup}
  (uniform version of Question \ref{quest:rk})

  As $K$ runs through all biquadratic fields, is the set
$
\bigcup_K (2\mathbb N - R_K)
$
finite?  In other words, does there exist an absolute constant
$
c_2>0
$
such that for every biquadratic field $K$, every even integer
$
\ge c_2
$
is the index of a lift of $\trho_K$?
\end{quest}

\begin{quest}
          \label{quest:cap}
          (a different uniform question)

%          
% As $K$ run through all biquadratic fields, is the set
% $
% \cap_K (2\mathbb N - R_K)
% $
% finite?  In other words,
%
Does there exist an absolute constant
$
c_3>0
$
such that every biquadratic field $K$ has a lift with index
$
\le c_3
$?
Equivalently, is the set
        $
        \{ \iota(K):  \text{$K$ biquadratic} \}
        $
bounded?
\end{quest}

\begin{quest}
  \label{quest:exist}
  As $K$ runs through all biquadratic fields, is the set
  $
  \bigcap_K (2\mathbb N - R_K)
  $
  empty?  In other words, given
$
m\in 2\mathbb N
$,
does there exist a biquadratic field with index $m$?
\end{quest}

Given an even integer $n$, write
$
n = n_1 2^k
$
with $n_1$ odd and $k\ge 1$.  We can show that if a biquadratic field has a
lift of index $n$, then it has a lift of index
$
2^k
$.
So the issue really comes down to the existence of lifts of $2$-power index.
Using the computer algebra system {\tt magma}
    \cite{magma} in conjunction with the
  LMFDB \cite{lmfdb},
  for each
$
m \in \{2, 4, 8\}
$
we have
found
examples of biquadratic fields with a lift of index $m$ but not of index $2$.
This raises
the possibility that the answers to the questions above could
be false.   In view of this, we can study \textit{asymptotic} forms of
these questions.

  \begin{quest}
     \label{quest:size}

(a) What upper bounds can one establish for the size of
$
\{ m\in 2\mathbb N:  m\in R_K \}  
$
and
$
\iota(K)
$
in terms of the discriminant of $K$?

(b)
As before, let
  $
  \mathcal{B}(X)
  $
denote the set of biquadratic fields of discriminant $\le X$.
What are the asymptotic behaviors of the sums
 \[
     \sum_{\substack{K\in \mathcal{B}(X)\\ \iota(K) = m }} 1
         \qquad
         \text{and}
         \qquad
     \sum_{K\in \mathcal{B}(X)} \iota(K)?
 \]
Note that our Theorem \ref{thm:density} gives an asymptotic for the first of these sums in the case $m=2$. 
\end{quest}

\begin{remm}
Let $K$ be a biquadratic field with discriminant
$
\Delta
$.
By the local analysis in \cite[p.~229]{serre:one}, the associated projective
representation
$
\trho_K
$
has a lift $\rho_K$ ramified at the same set of primes as $\trho_K$, such that
for each ramified prime 
$
\ell > 3
$,
the local character
$
\det \rho\bigr|_{D_\ell}
$
has order
$
\ll \ell^2
$.
Additionally, for $\ell\le 3$,  there are a bounded number of possibilities for
$
\det \rho\bigr|_{D_\ell}
$
for this minimal lift $\rho$.   Thus $\trho_K$ has a lift with index
$
\ll \prod_{\ell | \Delta} \ell^2 \ll \Delta
$; in other words, $\iota(K)\ll\Delta$. 
This gives an answer to the second part of Question
\ref{quest:size}(a).
Is it possible to sharpen this bound to, say, $\ll (\log \Delta)^A$ for some constant $A>0$?
\end{remm}

\if 3\
{
  , in which case we should modify question \ref{quest:three}
  and ask for an \textit{asymptotic upper} for
  $
  \iota(K)
  $
  in terms of the discriminant of $K$, and we should modify question
  \ref{quest:one} by asking for an upper bound of
  $
  \iota(K)
  $
  in terms of the discriminant of $K$.
 Last but not least, it is natural
  to ask if questions \ref{quest:one}-\ref{quest:three} are true for
  almost all biquadratic fields.
  }
\fi

So far we have been discussing \textit{quantitative} aspects of indices.
Next we turn to algebraic questions.

\begin{quest}
  \label{quest:index}

(a) Let $K$ be a biquadratic field $K$.  For every positive integer multiple
$m$ of
  $
  \iota(K)
  $,
does $\trho_K$ have a lift of index $m$?

(b) Suppose a biquadratic field $K$ has a lift of index $m$ and another lift of
index $n$. Does
$
\trho_K
$
necessarily have a lift of index $\gcd(m, n)$?
\end{quest}

  Kiming et al. have undertaken careful study of liftings of
  $2$-dimensional projective Galois representations
  \cite{kiming, kiming-2004}.
  In particular, \cite{kiming-2004} gives precise control over the
  \textit{local lifts} of projective dihedral representations.  Translating
  these local results to bounds on the \textit{global} index requires
  additional work.
We are currently investigating these questions; we will report on our results
  in another paper.

Last but not least, we can also study lifts of non-dihedral projective
Artin representations.

\begin{quest}
            \label{quest:exotic}
  (a)
  There are two possible index $2$ lifts for $S_4$-Artin representations.
  Is there an analog of Theorem \ref{thm:density} for these
  $S_4$-lifts?

  (b)
  There is only one possible index $2$ lift for $A_4$- and respectively
  $A_5$-projective Artin
  representations; what is the asymptotic density of such lifts?
\end{quest}

Let $K/\Q$ be a quartic (resp.~quintic) field with Galois group $A_4$ or
  $S_4$
(resp.~$A_5$).  Serre \cite{serre:traceform} has derived criteria for the splitting
field of $K$ to admit each of the index $2$ lifts in Question \ref{quest:exotic}.
On the other hand,
Bhargava \cite{bhargava} has proved an asymptotic formula for the number
of
$S_4$-quartic fields of bounded discriminants and with given signature.  It would be
very interesting to see if we can apply Serre's criterion in the context of Bhargava's enumerative techniques to study the
density of $S_4$-quartic fields with index $2$ lifts.

\if 3\
{
A major obstacle to studying index $2$ lifts of $A_4$- and $A_5$-quintic fields is
that at present there is no
  asymptotic formula for the number of $A_4$ or $A_5$ fields
  of bounded discriminants.
Since almost all integer polynomials of degree $n$ have $S_n$-Galois
  group, in the case of $S_4$-lifts, a possible substitute would be to
  work with degree $4$ monic integer polynomials of bounded height.
}
\fi
We will address these and other questions in another paper.

% {\tt application:  short sum approximation and effective chebotarev}

\subsection{Relationship with the works of Rohrlich}
So far we have been focusing on the index, a group theoretic invariant of lifts.
Artin representations have another important invariant called
\emph{Artin conductors}; they are defined in terms of ramification groups of the
splitting fields of the representations and play crucial roles in algebraic
number theory. In a series of
papers
   \cite{rohrlich-dih-type, rohrlich-quat, rohrlich-quat2, rohrlich-dihedral},
Rohrlich
investigates quantitative questions about 
irreducible Artin representations of bounded Artin conductor that are lifts
of dihedral projective Artin
representations.
We will now state the results of Rohrlich and explore the relationships of these results to ours.

For any integer $m\ge 2$, denote by
\[
   Q_{4m}
   =
   \myvec{ a, b:  a^{2m} = 1, a^m = b^2, bab^{-1} = a^{-1} }
\]
the generalized
quaternion
    group\footnote{The name of these groups and the notation
                   $
                   Q_{4m}
                   $
                   are taken verbatim from \cite{rohrlich-quat}.  Rohrlich also
                   notes that a more common name for these groups are
                   \textit{dicyclic groups} (notation:  
                   $
                   \text{Dic}_m
                   $),
                   and that in many group theory texts, the name
                   \textit{generalized quaternion groups} is reserved for
                   $
                   \text{Dic}_m
                   $
                   only when $m>1$ is a power of
                   $2$.
                   We stick to Rohrlich's name of the groups $Q_{4m}$ as these
                   groups arise in this paper only in the context of comparing
                   his works with ours.
         }
of
order $4m$; for $m=2$ this is the usual quaternion group of order
$8$.
It is not isomorphic to the dihedral group
$
D_{2m}
$,
since $b^2$ is the unique involution of $Q_{4m}$ \cite[p.~63]{feit}.  However,
$
D_{2m}
$
and
$
Q_{4m}
$
have the same character table \cite[p.~64]{feit}; in particular, the
irreducible, non-linear characters of
$
Q_{4m}
$
all have degree $2$.  These degree $2$ characters fall into two
types \cite{group}:
  there are the ones that factor through 
$
Q_{4m}/Z(Q_{4m})
$,
which is a dihedral group; in particular, these quotient characters are all
orthogonal.  The remaining degree $2$ characters of
$
Q_{4m}
$
are faithful and quaternionic.

We say that 
$
\rho: G_\Q \to GL_2(\mathbb C)
$
is dihedral (resp.~quanternionic) if its image is a dihedral group
(resp.~a generalized quaternion group).  Using Siegel's asymptotic class number
formulas, Rohrlich proves that
\cite{rohrlich-dih-type}
\begin{equation}
  \Bigl(
    \begin{array}{l}
      \text{number of dihedral $\rho$ with}
      \\
      \text{Artin conductor $N(\rho) \le Y$}
    \end{array}
  \Bigr)
  \;\sim\;
  \frac{\pi}{36 \zeta(3)} Y^{3/2}.
                  \label{dih}
\end{equation}
Additionally, for every
$
\varepsilon > 1/4\sqrt{e}
$,
he proves that   \cite{rohrlich-quat} 
\begin{equation}
  Y^{1-\varepsilon}
  \;\ll_{\varepsilon}\;
  \Bigl(
    \begin{array}{l}
      \text{number of quaternionic $\rho$ with}
      \\
      \text{Artin conductor $N(\rho) \le Y$}
    \end{array}
  \Bigr)
  \;\ll\;
  \frac{Y}{\log Y}.
                  \label{quat}
\end{equation}
The projective representation
$
\trho
$
associated to a dihedral or quaternionic $\rho$ is biquadratic.  Conversely,
given a biquadratic field of discriminant
$
\Delta
$,
the corresponding projective representation has a lift with Artin conductor
dividing
$
c_0 \Delta^2
$
for some absolute constant
$
c_0>0
$
\cite[p.~229]{serre:one}.  Setting
$
Y = c_0 \Delta^2
$
in (\ref{dih}) and (\ref{quat}) then gives an upper bound for the number of
biquadratic fields of bounded discriminant that admit $D_4$- and respectively
$Q_8$-extensions.  This upper bound is weaker than Theorem \ref{thm:density},
since (\ref{dih}) and (\ref{quat}) impose no restriction on the index of the
lifts. Thus Rohrlich's  results do not replace our results, and vice versa.

\if 3\
{
Rohrlich's works do have connections with the questions raised in section
\ref{sec:future}.
We noted earlier than  every biquadratic field of discriminant
$
\Delta
$
has a lift $\rho$ with Artin conductor dividing
$
c_0 \Delta^2
$.
}
\fi

\if 3\
{
If we denote
by
$
S(\trho)
$
and
$
S(\rho)
$
the set of primes that ramify in $\trho$ and $\rho$, then
$
S(\trho) \subseteq S(\rho)
$,
and the discriminant of the splitting field
$
K_\trho
$
of
$
\trho
$
is a bounded power of
$
\prod_{\ell\in S(\trho)} \ell
$.
That means this biquadratic field discriminant is bounded by a
power of
$
N(\rho)
$.
Consequently, the biquadratic fields with $D_4$- and respectively $Q_8$-lifts
in theorem \ref{thm:density} are contained in the set

To understand the connection between  Rohrlich's estimates with our theorem
\ref{thm:density}, we need to relate Artin conductor with biquadratic

Conversely, for any biquadratic
$
\trho
$,
there exists a lift
$
\rho_0: G_\Q\rarr GL_2(\mathbb C)
$
\[
   N(\rho_0) = \min\{ N(\rho):  \text{$\rho$ is a lift of $\trho$} \}.
\]
This implies that  $\rho_0$ is ramified precisely at the ramified primes
of
$
\trho
$.
This minimal conductor can be determined entirely from
$
\trho
$:
For any ramified prime $\ell >3$, the exponent of $\ell$ in
$
N(\rho)
$
is either $1$ or $2$, depending on whether the decomposition group of $\ell$
in
$
\trho
$
is cyclic or dihedral \cite[p.~229]{serre:one}. The case $\ell\le 3$ is a
finite calculation; in particular, there are finitely many possibilities.
It follows that given a biquadratic field of discriminant
$
\Delta
$,
there exists a lift with Artin conductor
$
\ll \Delta^2
$.
}
\fi

\begin{ack}
The authors are grateful to Professors Robert J. Lemke Oliver and Igor Shparlinski for their careful reading of a draft, for making the authors aware of several relevant references, and for insightful discussions.
\end{ack}

\bibliographystyle{amsalpha}

\vfill
\hrule

\end{document}